\documentclass[bibother, otherbib]{asl}

\usepackage[usenames,dvipsnames]{xcolor}
\usepackage{enumerate}
\usepackage{manfnt}
\usepackage[shortlabels]{enumitem}
\usepackage{makecell}
\usepackage{tikz-cd}
\usepackage{hyperref}
\hypersetup{colorlinks=true,linkcolor=blue,urlcolor=blue,citecolor=blue}

\title{Locally tabular products of modal logics}


\author{Ilya B. Shapirovsky}
\revauthor{Shapirovsky, Ilya B.}

\address{Department of Mathematics\\
New Mexico State University\\
1780 E University Ave, Las Cruces, NM 88003, USA}

\email{ilshapir@nmsu.edu}

\thanks{
The work of the first author was supported by NSF Grant DMS - 2231414.% \IS{Move it}
}

\author{Vladislav V. Sliusarev}
\revauthor{Sliusarev, Vladislav V.}

\address{Department of Mathematics\\
New Mexico State University\\
1780 E University Ave, Las Cruces, NM 88003, USA}

\email{vnvdvc@nmsu.edu}
\newtheorem{theorem}{Theorem}[section]
\newtheorem{proposition}[theorem]{Proposition}
\newtheorem{lemma}[theorem]{Lemma}
\newtheorem{corollary}[theorem]{Corollary}

\newtheorem{claim}[theorem]{Claim}

\theoremstyle{definition}
\newtheorem{definition}[theorem]{Definition}
\newtheorem*{definition*}{Definition}
\newtheorem{remark}[theorem]{Remark}
\newtheorem{example}[theorem]{Example}

\newcommand\hide[1]{{\empty}}

\newcommand\ISH[1]{{\bf IS}: {\color{teal} #1}}
\renewcommand\ISH[1]{\empty} 
\newcommand\IS[1]{\ISH{#1}}
\newcommand\ISLater[1]{{\bf IS}: {\color{red} #1}}
\renewcommand\ISLater[1]{\empty} 
\newcommand\VS[1]{{\bf VS}: {\color{blue}#1}}
\renewcommand\VS[1]{\empty}
\newcommand\todo[1]{ [~ {\color{BrickRed} #1 }]}

\newcommand\extended[1]{ {\color{BlueViolet} Extended:~ #1 }]}
\renewcommand\extended[1]\empty

\newcommand\improve[1]{ [~ {\color{BlueGreen}\noindent{\bf Improve:} #1 }]}
\renewcommand\improve[1]\empty

\def\Al{\mathrm{A}}
\def\AlA{\Al}
\def\AlB{\mathrm{B}}

\newcommand\framets[1]{#1}
\def\frI{\framets{I}}
\def\frF{\framets{F}}
\def\toto{\twoheadrightarrow}
\newcommand\LSum[2]{{\textstyle \sum_{#1}{#2}}}

\def\tiff{\text{ iff }}
\def\clI{\mathcal{I}}
\def\clX{\mathcal{X}}
\def\clY{\mathcal{Y}}
\def\clH{\mathcal{H}}
\def\clV{\mathcal{V}}
\def\clC{\mathcal{C}}
\def\clF{\mathcal{F}}
\def\clG{\mathcal{G}}
\def\clR{\mathcal{R}}
\def\Log{\myoper{Log}}
\newcommand\myoper[1]{\mathop{\myopts{#1}}}
\newcommand\myopts[1]{\mathrm{#1}}
\def\Di{\lozenge}

\def\DiAl{\Di_\Al}
\def\imp{\rightarrow}

\def\GL{\LogicNamets{GL}}
\newcommand\LogicNamets[1]{\logicts{#1}}
\newcommand\logicts[1]{{\textsc{#1}}}
\newcommand\LS[1]{\LogicNamets{S#1}}
\newcommand\LK[1]{\LogicNamets{K#1}}
\newcommand\Grz{\LogicNamets{Grz}}
\def\vL{L}

\def\clU{\mathcal{U}}
\def\clW{\mathcal{W}}
\def\clV{\mathcal{V}}

\def\clC{\mathcal{C}}
\def\clD{\mathcal{D}}
\def\clK{\mathcal{K}}
\def\clV{\mathcal{V}}
\def\clS{\mathcal{S}}
\def\clT{\mathcal{T}}
\def\clP{\mathcal{P}}
\def\EE{\exists}
\def\AA{\forall}

\def\dom{\myoper{dom}}
\def\restr{{\upharpoonright}}

\def\v{\theta}
\def\vext{\bar{\v}}

\newcommand\tranA[1]{{ [#1 ]}^{m}}
\newcommand\tranAk[1]{{ [#1 ]}^{k}}
\newcommand\tran[1]{{ [#1 ]}^{m,n}}
\newcommand\traone[1]{[#1]^\star}

\def\vf{\varphi}
\def\mo{\vDash}
\def\vd{\vdash}
\def\con{\wedge}
\def\lra{\leftrightarrow}
\def\emp{\varnothing}
\def\SubFrs{\myoper{Sub}}

\newcommand\languagets[1]{\logicts{#1}}
\def\ML{\languagets{ML}}
\def\PV{\languagets{PV}}

\newcommand{\quotfr}[1]{\widetilde{#1}}

\newcommand{\rect}[2]{\boldsymbol#1 \times \boldsymbol#2}

\newcommand\clusters[1]{\framets{Cl}{(#1)}}
\newcommand\Sk{\myoper{Sk}}
\newcommand{\inv}{^{-1}}

\newcommand\un[1]{{#1}_\logicts{u}}

\newcommand\Tsum[1]{{{#1}\,\oplus\, \circ}}

\def\TL{\LogicNamets{Tack}}
\def\TF{\textsc{T}}
\def\LSaw{\LogicNamets{Saw}}

\newcommand\fact[2]{{#1}{/}{#2}}
\newcommand{\bomega}{\boldsymbol\omega}

\def\FORP{\mathrm{RP}}
\def\RP{\mathrm{rp}}
\newcommand\tra[1]{\mathrm{tra}{(#1)}}

\newcommand{\fusion}[2]{#1*#2}
\newcommand{\LCom}[2]{\left[#1,#2\right]}
\def\com{\mathrm{com}}
\def\chr{\mathrm{chr}}
\begin{document}
\begin{abstract}
In the product $L_1\times L_2$ of two Kripke complete consistent logics, local tabularity of $L_1$ and $L_2$ is necessary  for local tabularity of $L_1\times L_2$. 
However, it is not sufficient:  the product of two locally tabular logics may not be locally tabular.  We provide extra semantic and axiomatic conditions that give criteria of local tabularity of the product of two locally tabular logics, and apply them to identify new families of locally tabular products. We show that the product of two locally tabular logics may lack the product finite model property.  We give an  axiomatic criterion of local tabularity  for all extensions of 
$\LS{4}.1 [ 2 ]\times \LS{5}$.  Finally, we describe  a new prelocally tabular extension of $\LS{4}\times\LS{5}$.  
\end{abstract}

\maketitle

\section{Introduction} ~

It is well-known that the operation of product of modal logics does not 
preserve the finite model property of the factors, see, e.g.,  \cite{ReynoldsZakh_Linear2001},\cite{GabelaiaAtAlUndec2005}, or the monography 
\cite{ManyDim}. In this paper we describe new families of modal products  which have the finite model property and, in fact, satisfy a stronger property of  local tabularity.

\ISLater{Products: explain what are their. Explains the FMP issue

The product of modal logics defined semantically...
(Pick definition of product from GSS - it is good; and motivation from book) 

}

A logic is locally tabular, if each of its finite-variable fragments contains only a finite number of pairwise nonequivalent
formulas.  In particular, every locally tabular logic has the finite model property. 
It is well known that for unimodal logics above $\LK{4}$, local tabularity is equivalent to finite height \cite{Seg_Essay},\cite{Maks1975LT}. 
\extended{
This criterion was extended for weaker systems in \cite{LocalTab16AiML}: it holds for logics containing $\Di^{k+1} p\imp \Di p\vee p$. }
In the non-transitive unimodal, and in the polymodal case, no axiomatic criterion of local tabularity is known.

It follows from \cite{MLTensor} that
the product of a locally tabular logic with a tabular  one is locally tabular. 
However, these cases are not exhaustive: 
other families of locally tabular modal products were identified
in \cite{Shehtman2012}, see also \cite{Shehtman2018}. For close systems, {\em intuitionistic modal logics} and {\em  expanding products}, 
locally tabular families were identified in \cite{Guram98MIPCI}, \cite{Guram-Revaz}, \cite{Bezhanishvili2001}, and a recent manuscript \cite{Bezhanishvili2023local}.

In the product $L_1\times L_2$ of two Kripke complete consistent logics, local tabularity of $L_1$ and $L_2$ is necessary  for local tabularity of $L_1\times L_2$.
%: the product is conservative over the factors in this case. 
However, it is not sufficient: 
the product of two locally tabular logics can be not locally tabular. 
The simplest example is the logic $\LS{5}\times \LS{5}$  \cite{CylindricalAlgebras1}. 
We provide extra semantic (Theorem \ref{thm:criterion-frames-general}) and axiomatic (Corollary \ref{cor:criterion-logics-general}) conditions which give criteria of local tabularity of the product of two locally tabular logics: 
bounded cluster property of one of the factors; 
a condition we call {\em product reducible path property}; finiteness of the one-variable fragment of the product.

In Section \ref{sec:examples}, we apply the criteria to identify new families of locally tabular products. In particular, we %generalized the Segerberg-Maksimove criterion for some classes of products (?)  and
generalize some results from \cite{Shehtman2012,Shehtman2018}.

In Section \ref{sec:pfmp}, we discuss the product finite model property in the locally tabular case. 
A modal logic \(L\) has the \emph{product fmp}, if \(L\) is the logic of a class of finite product frames.
 The product fmp is stronger than the fmp: for example, \(\LK4\times \LS5\) has the fmp~\cite{GabbayShehtman-ProductsPartI}, but lacks the product fmp~\cite{ManyDim}.  
It is perhaps surprising that the local tabularity of a product logic does not imply the product fmp even in the case of height 3, as we discovered  in  Theorem \ref{prop:saw_logic_no_pfmp}.

In Section \ref{sec:ProdaboveS5}, we discuss local tabularity 
and prelocal tabularity in extensions of $\LS{4}\times\LS{5}$. 
The logic $\LS{5}$
is known to be one of the five pretabular logics above $\LS{4}$ \cite{EsakiaMeskhi1977}, \cite{MaksimovaPretab75}. 
We observe that for another pretabular logic $\TL$, the logic of the ordered sum of a countable cluster and a singleton, the product $\TL\times \LS{5}$ is not prelocally tabular. 
Then we consider a  weaker than $\TL$ logic $\LS{4}.1 [ 2 ]$, the extension of $\LS{4.1}$ with the axiom of height 2, 
and  give an axiomatic criterion of local tabularity for all normal extensions of $\LS{4}.1 [ 2 ]\times \LS{5}$. 
Finally, we discuss prelocal tabularity. 
It is known that $\LS{5}\times \LS5$ is prelocally tabular \cite{NickS5}. 
We construct a frame of height two, which we call {\em two-dimensional tack}, and show that its logic 
is another example of a 
 prelocally tabular logic above $\LS{4}\times \LS{5}$.

 \hide{
 In particular, in Section \IS{Remove}, 
we discuss prelocal tabularity of $\LS{5}\times \LS{5}$ \cite{NickS5}.  
We show that for a Kripke complete proper extension  of $\LS{5}\times \LS{5}$, 
local tabularity is a simple corollary of Theorem \ref{thm:criterion-frames-general}. 
\improve{\ISH{Do we give a shorter proof?}}

...
}

\section{Preliminaries}\label{sec:prel}

For basic notions in modal logic, see, e.g., \cite{CZ} or \cite{BDV}.

\subsection{Modal syntax and semantics.} Let $\Al$ be a finite set, an {\em alphabet of modalities}.
%We will always assume that $\Al$ is finite.
{\em Modal formulas over $\AlA$}, $\ML(\AlA)$ in symbols, are constructed from
a countable set of {\em variables} $\PV=\{p_0,p_1,\ldots\}$ using Boolean connectives $\bot,\imp$ and unary connectives $\Di\in \AlA$.
Other logical connectives are defined as abbreviations in the standard way, in particular $\Box\vf$
denotes $\neg \Di \neg \vf$. 

%The term {\em unimodal} refers to the case when $\AlA$ is a singleton.

We define the following abbreviations: $\Di^0 \vf=\vf$,
$\Di^{i+1}\vf=\Di^i\Di\vf$,
$\Di^{\leq m} \vf =
\bigvee_{i\leq m} \Di^i \vf$, $\Box^{\leq m} \vf =\neg \Di^{\leq m} \neg \vf$.
We write $\DiAl\vf$ for $\bigvee_{\Di\in\Al}\Di \vf$. 

By an {\em $\AlA$-logic} we mean a normal modal logic whose alphabet of modalities is $\AlA$, see, e.g. \cite{BDV}. The terms {\em unimodal} and {\em bimodal} refers to the cases $\AlA=\{\Di\}$ and \(\AlA = \{\Di_1,\Di_2\}\), respectively.

An {\em \(\AlA\)-frame} is a pair \(F = (X,\,(R_\Di)_{\Di\in \AlA})\), where \(X\) is a non-empty
set and \(R_\Di \subseteq X\times X\) for any \(\Di\in \AlA\). We write \(\dom F\)  for  \(X\). 
By the cardinality $|F|$ of $F$ we mean the cardinality of $X$.
We put \(R_F = \bigcup_{\Di\in \AlA} R_\Di.\) 
For $a\in X$, $Y\subseteq X$, we put $R_\Di(a)=\{b\mid aR_\Di b\}$,
$R_\Di[Y]=\bigcup_{a\in Y} R_\Di(a)$.

A {\em model on} $F$ is a pair \((F,\,\theta)\), where \(\theta:\:PV \to \clP(X)\), and $\clP(X)$ is the powerset of \(X\); $\theta$ is called a {\em valuation in \(F\)}.  
The truth-relation  $(\frF,\theta),a\mo \vf$ is defined in the standard way; in particular, for $\Di\in \AlA$,  
$(F,\,\theta),\,a\models \Di \varphi$ means that $(F,\,\theta),\,b\models \varphi\text{ for some }b\in R_\Di(a)$. We put
$$\vext(\vf)=\{a\mid (F,\v),a\mo\vf\}.$$
A formula $\vf\in \ML(\AlA)$ is {\em valid in an  $\Al$-frame $\frF$}, in symbols $\frF\mo\vf$,
%if $M,a\mo\vf$ for every model $M$ on $F$ and every $a$ in $F$. 
if $X=\vext(\vf)$ for every model $(F,\theta)$ on $F$.
For a set of formulas $\Gamma\subseteq \ML(\Al)$, $F\mo\Gamma$ means that $F\mo\vf$ for all $\vf\in\Gamma$; 
in this case $F$ is said to be a {\em $\Gamma$-frame}.
For a class $\clF$ of $\Al$-frames $\clF$, $\clF\mo\Gamma$
means that $F\mo \Gamma$ for all $F\in\clF$.

The set of \(\AlA\)-formulas that are valid in \(F\) is denoted by
\(\Log F\). 
For a class $\clF$ of $\AlA$-frames,  \(\Log \clF=\bigcap\{\Log F \mid F\in \clF \}\);
this set is a logic \cite{BDV}\improve{\ISH{Specify?} } and is called the {\em logic of $\clF$}.  Such logics are said to be {\em Kripke complete}. 
\extended{ For a set of formulas \(\Gamma \subseteq \ML(\AlA),\) \(\Frames(\Gamma)\) is the class of \(\AlA\)-frames where \(\Gamma\) is valid.
}

Let us recall that \(\LK{4}\) is the smallest unimodal logic that contains $\Di \Di p \to \Di p$,
$\LS{4}$ extends \(\LK{4}\) with $p\imp \Di p$, and $\LS{5}$ extends $\LS{4}$ with $p\imp\Box\Di p$. 
%Let us recall that \(\LS{4}\) is the smallest unimodal logic that contains \(p \to \Di p,\,\Di \Di p \to \Di p.\) 
These logics are Kripke complete; $\Di \Di p \to \Di p$ expresses the transitivity, 
$p\imp \Di p$ reflexivity, and $p\imp\Box\Di p$ the symmetry of a binary relation, see, e.g., \cite{BDV} or \cite{CZ}.\improve{specify}

 \hide{
It is a standard fact that \(\LS{4}\) is Kripke complete, and every frame that validates~\(\LS{4}\) is a preorder. The unimodal logic \(\LS{5}\) is the extension of \(\LS{4}\) obtained by adding \(p \to \Box\Di p,\) which defines the class of equivalence relations \cite{CZ}.
}

The notions of {\em generated subframe}\hide{, {\em rooted frame},} and {\em p-morphism} (or {\em bounded morphism})  are defined in the standard way, see, e.g., 
\cite[Section 3.3]{BDV}. In particular, 
  for frames \(F = (X,\,(R_\Di)_{\Di\in \AlA})\) and \(G = (Y,\,(S_\Di)_{\Di\in \AlA})\), 
 \(f:\:X\to Y\) is a {\em p-morphism from \(F\) to \(G\)}, if   
  \begin{enumerate}
    %\item surjective;
    \item $f$ is a homomorphism, that is  \hide{monotone:} for any \(a,\,b\in X\) and \(\Di\in \AlA,\) \(a R_\Di b\) implies \(f(a) S_\Di f(b)\), and  
    \item $f$ satisfies the {\em back condition}, that is
     %\ISH{lifting - is it a common term?}  
     for any \(a\in X,\,u\in Y\), and \(\Di\in \AlA,\) \(f(a) S_\Di u\) implies that there exists \(b\in X\) such that \(a R_\Di b\) and \(f(b) = u.\) 
  \end{enumerate}
  If $f$ is surjective,  we write \(f:\:F\toto G.\) If there exists a p-morphism from \(F\) onto~\(G,\) we write \(F \toto G;\) in this case, \(\Log{F} \subseteq \Log{G}\)
  \cite[Section 3.3]{BDV}.

For a frame \(F = (X,\,(R_\Di)_{\Di\in A})\) and an equivalence $\approx$ on $F$, 
the {\em quotient frame} $\fact{F}{\approx}$ is the frame  
$(\fact{X}{\approx},\,(R^\approx_\Di)_{\Di\in\Al})$, where for $[a],[b]\in  \fact{X}{\approx}$ and $\Di\in \AlA$,
  \[
    [a] R^\approx_\Di [b] \tiff \exists a'\in[a]\ \exists b'\in [b]\ \left(a' R_{\Di} b' \right).
  \]

\improve{Give ref to \cite{ManyDim} below}
\subsection{Products.}
Let $\AlA$ and $\AlB$ be two disjoint finite sets. 
%\ISH{Later: say about the case $\AlA=\AlB$ that we shift modalities} 
%\newcommand\prod[1]{{#}
% \ISH{alternative version of product}
For an \(\AlA\)-frame \(F=\left(X,\,(R_\Di)_{\Di\in \AlA}\right)\) and \(\AlB\)-frame \(G = \left(Y,\,(R_\Di)_{\Di\in \AlB}\right)\), the {\em product frame} \( F\times G\) is the frame $(X\times Y,\,(R^h_\Di)_{\Di\in \AlA}, (R^v_\Di)_{\Di\in \AlB})$,  where 
\begin{align*}
    R^h_\Di&=\{((a,\,b),\,(a',\,b))\mid a,\,a'\in X,\,b\in Y,\,a R_\Di a'\}  &\text{ for }\Di\in \AlA;\\
    R^v_\Di &= \{((a,\,b),\,(a,\,b'))\mid a\in X,\,b,\,b'\in Y,\,b R_\Di b'\}  &\text{ for }\Di\in \AlB.
\end{align*} 
We say that relations $R^h_\Di$ are {\em horizontal},  and $R^v_\Di$ are {\em  vertical}. 
\hide{
\ISH{What is more convenient in proofs: horizontal/vertical or induced?}
so $F\times G$  has $|\AlA|$ horizontal and $|\AlB|$ vertical relations.

\hide{
The product frame \(F\times G\) is an \((\AlA\sqcup\AlB)\)-frame.
\ISH{Earlier, $a,b$ were used for elements of frames. Now we have different notation. 
I think it is OK (if only we do not use first-order formulas extensively), but should be consitent through the text.} 
}

We use the notation \(R^h\) for \(\bigcup_{\Di\in \AlA} R_\Di^h\) and \(R^v\) for \(\bigcup_{\Di\in \AlB} R_\Di^v\)
\ISH{When possible, let us minimize the conventions. Do we need $R^h, R^v?$}

%\ISH{Or: $R^\times_\Di$ for product relations?}
}
For a class \(\clF\) of \(\AlA\)-frames and a class \(\clG\) of \(\AlB\)-frames, we put
%\[
$
  \clF \times \clG = \{F\times G \mid F\in \clF,\,G\in\clG\}.
  $
%\]
\improve{Is is a good style when the set builder 
$\{x \mid \phi(x) \and \phi(x)\}$ is abbreviated as
$\{x \mid \phi(x), \phi(x)\}$? It is quite common, but is it common for  good manuscripts?}

For logics $L_1,L_2$, the {\em product logic \(L_1\times L_2\)} is defined as
\[
  L_1\times L_2 = \Log\{F\times G \mid F\models L_1,\,G\models L_2\}.
\]
If $L_1$ and $L_2$ are unimodal logics, we rename their modalities and follow the convention that
the alphabet of $L_1\times L_2$ is $\{\Di_1,\Di_2\}$.  
We write $L^2$ for  $L\times L$.

\subsection{Local tabularity and local finiteness.}
For finite $k$, a logic $L$ is {\em $k$-finite}, if
$L$ contains only a finite number of pairwise nonequivalent
formulas in variables $p_i$, $i<k$. \extended{ Equivalently, $L$ is $k$-finite, if
the $k$-generated Lindenbaum algebra of $L$ is finite. }
$L$ is said to be {\em locally tabular}, if it is $k$-finite for all finite $k$.
The following fact is well known, see, e.g., \cite[Chapter 12]{CZ}.\improve{Check the ref}
\begin{proposition}\label{prop:LT_implies_fmp+extensions}~ Let $L$ be locally tabular. Then:
\begin{enumerate}[1.] 
\item $L$ has the {\em finite model property}, that is, $L$ is the logic of a class of finite frames. In particular, $L$ is Kripke complete. 
    \item  Every extension of $L$ (in the same modal alphabet) is locally tabular.
\end{enumerate} 
\end{proposition}

In algebraic terms, local tabularity of a logic  means that its 
Lindenbaum algebra is locally finite, which, in turn, means  that the  variety of its algebras is locally finite (see, e.g., \cite[Chapter 2]{BurrisSankappanavar2012}, \cite[Chapter 12]{CZ} for corresponding  notions).  

\smallskip 
A class  $\clC$ of algebras of a finite signature is said to be
{\em uniformly locally finite}, if
there exists a function $f:\omega\to \omega$ such that the cardinality of a subalgebra of any $B\in \clC$ generated by
$k<\omega$ elements does not exceed $f(k)$.
\begin{theorem}\cite[Section 14, Theorem 3]{Malcev73}\label{Malcev73}.  
Local finiteness of the variety %$\clK$ 
generated by a class $\clC$
is equivalent to uniform local finiteness of $\clC$. 
\end{theorem}

\subsection{Pretransitivity and finite height.}

For a binary relation $R$ on $X$ and a number \(m < \omega\), put
$R^{\leq m}=  \bigcup_{i \leq m} R^i$, where
$R^{i+1}=R\circ R^i$, $\circ$ is the composition of relations, $R^0$ is the diagonal relation on $X$. The transitive reflexive closure \(\bigcup_{i < \omega} R^i\) of $R$ is denoted by $R^*$.

A frame $F$ is said to be {\em $m$-transitive},  if $R_F^{\leq m}=R_F^*$.\footnote{
To avoid any ambiguity, we remark that we write $R_F^{\leq m}$ for $(R_F)^{\leq m}$, and $R_F^*$ for   $(R_F)^*$.
%The notation $R_F^*$.$\,^1$ refers to a footnote. 
}
$F$ is 
{\em pretransitive}, if it is $m$-transitive for some finite $m$.
\hide{
It is straightforward that
\begin{equation}
\text{$R$ is $m$-transitive iff
$R^{m+1}\subseteq R^{\leq m}$.  }
\end{equation} 
\extended{
A frame \(F = (X,\,(R_\Di)_{\Di\in \AlA})\) is {\em rooted} if there exists a point \(r\in X,\) called {\em the root of \(F\)}, such that \(X = R_\Di^*(r)\) for any \(\Di\in \AlA.\)
\ISH{It seems that something is wrong here}
}
}
\ISLater{Do we need it?}
The following fact is standard.
\begin{proposition}[Jankov-Fine theorem for pretransitive frames]\label{prop:Jankov-Fine}
If $\Log G\subseteq \Log F$, $G$ is pretransitive, and $F$ is finite and rooted,
%\ISH{terminology?}
then there exists a point-generated subframe $H$ of $G$ such that $H\toto F$.
% \VS{Do we need the definition of \(\toto\)? The p-morphism lemma?}
% \ISH{Yes, to address  properties of p-morphism (homomorphism, back property) in proofs; p-morphism lemma as a separate statement is not required, can be just mentioned in text}
\end{proposition} 

%\todo{$R_F$}

For any $\Al$-frame $\frF$, we have the following equivalence \cite[Section 3.4]{KrachtBook}:
\begin{equation}\label{pretr:sem}
\text{
$R_\frF$ is $m$-transitive iff
$\frF\mo \DiAl^{m+1} p \imp \DiAl^{\leq m} p$. }
\end{equation} 
A logic $\vL$ is said to be {\em $m$-transitive}, if
$\vL$ contains $\Di_\AlA^{m+1} p \imp \Di_\AlA^{\leq m} p$, and {\em pretransitive}, if it is $m$-transitive for some $m\geq 0$.

For a frame $F$, let $F^*$ be the preorder $(\dom{F},R_F^*)$. For a class $\clF$ of frames, $\clF^*=\{F^*\mid F\in F\}$.

The following fact is straightforward: for an $m$-transitive frame $F=(X,(R_\Di)_{\Di\in \Al})$, %and any $\Al$-formula $\vf$, we have %ERROR
and any unimodal formula $\vf$, we have
\begin{equation}\label{eq:pretrans}
F\mo \tranA{\vf} \text{ iff } F^*\mo \vf. %(X,R_F^*)\mo \tranA{\vf},
\end{equation}
where $\tranA{\vf}$ is the translation that is compatible with Boolean connectives (that is, 
$\tranA{\bot}=\bot$, $\tranA{p}=p$ for variables, $\tranA{\psi_1\imp \psi_2}=\tranA{\psi_1}\imp \tranA{\psi_2}$),
given by $\tranA{\Di \psi}= \Di_\AlA^{\leq m}  \tranA{\psi}$.

For a frame \(F = (X,\,(R_\Di)_{\Di\in \AlA}),\) let $\sim_F$ be the equivalence $R_F^*\cap (R_F^*)\inv$ on $X$, and   $\Sk F=\fact{F^*}{\sim_F}$.
The quotient frame $\Sk F$ is a poset; it is called the {\em skeleton of } $F$. 
%Elements of $\Sk F$, that is,  
Equivalence classes modulo  $\sim_F$  
are called {\em clusters in $F$.} 
%Recall that a poset is of \emph{height \(h <t \omega\)} if \(h\) contains a chain of \(h\) elements
%and no chains of more than \(h\) elements.

We say that $F$ {\em has a finite height $h$}, in symbols $h(F)=h$, %\VS{\(h(F) = h\)?},
if $\Sk F$ contains a chain of \(h\) elements
and no chains of more than \(h\) elements.
%If \(\clF\) is a class of transitive frames and \(1\le h < \omega\), \(h(\clF) \le h\) if and only if \(\clF \models B_h,\) where the formulas \(B_h\) are 
%defined as follows: 
%constructed inductively:
Consider unimodal formulas 
\[
  B_0 = \bot,\quad B_{i+1} = p_{i+1}\imp \Box \left (\Di p_{i+1} \lor B_i\right).%,\ 1 < i < \omega.
\] 
For a unimodal frame $G=(X,R)$ with a transitive $R$, we have \cite{Seg_Essay}: 
$$G\mo B_h \tiff h(G)\leq h.$$
Consequently, if \(F\) is an \(m\)-transitive $\Al$-frame,
then  \(F\models \tranA{B_h}\) iff \(h(F)\le h\).
Formulas $B_h$ and their polymodal pretransitive generalizations $\tranA{B_h}$ are called {\em formulas of finite height}.
For an \(m\)-transitive logic~\(L\), we define the \emph{height of \(L\)}, denoted by \(h(L),\) as the smallest \(h < \omega\) such that \(L \vdash [B_h]^m.\)

\begin{theorem}\cite{LocalTab16AiML}  \label{thm:1-finite-to-m-h}
Let $L$ be an $\AlA$-logic. 
If $L$ is 1-finite, then
for some $h,m<\omega$, $L$ contains $\Di_\AlA^{m+1} p \imp \Di_\AlA^{\leq m} p$ and
$\tranA{B_h}.$
\end{theorem}
In \cite{LocalTab16AiML}, this fact was stated for the unimodal case; its polymodal generalization  is straightforward, details can be found in \cite{LTViaSums2022}.

\subsection{Local tabularity and tuned partitions.} 

A {\em partition $\clV$ of a set $X$} is a family
of non-empty pairwise disjoint sets such that $X=\bigcup \clV$.
For $x\in X$, $[x]_\clV$ denotes $V\in\clV$ that contains $x$. 
A  {\em partition of $X$ induced by $h :X\to Y$} is the set of non-empty sets $h^{-1}(y)$, $y\in Y$.
%Equivalently, a partition is the quotient set of $X$ by an equivalence; this equivalence is %denoted by
%$\sim_\clV$.
A partition $\clU$ {\em refines} $\clV$, if each element of $\clV$ is the union of some elements of $\clU$.%,  equivalently, $\sim_\clU\;\subseteq \;\sim_\clV$.
%For a family $\clS$ of subsets of $X$, the {\em partition induced} by $\clS$
%is the quotient set
%$X/{\sim}$, where $$a\sim b \text{ iff }  \AA Y\in\clS \, (a\in Y \Leftrightarrow b\in Y).$$

\begin{definition}%\cite{Franz-Bull}\cite{LocalTab16AiML}
\label{def:tune}
Let $R$ be a binary relation on $X$.
A partition $\clU$ of $X$ is said to be {\em $R$-tuned}, if for every $U,V\in \clU$,
\begin{equation}\label{eq:part}
\EE a\in U  \EE b\in V \, (aR b)  \,\text{ implies }\, \AA a\in U  \EE b\in V \, (aR b).
\end{equation} 
Let $\frF=(X,(R_\Di)_{\Di\in \Al})$ be an $\Al$-frame.
A partition $\clU$ of $X$ is said to be {\em tuned in $\frF$}, if it is $R_\Di$-tuned for every $\Di\in \Al$.
In other terms, the condition $\eqref{eq:part}$  means that the map $x\mapsto [x]_\clU$ is a p-morphism  $F\toto \fact{F}{\sim}$, where $\clU=\fact{X}{\sim}$, see, e.g., \cite[Proposition 3.2]{Blok1980} for details.

\extended{
The frame $\frF$ is said to be {\em tunable}, if for
 every finite partition $\clV$ of $\frF$ there exists a finite tuned refinement $\clU$ of $\clV$.}
A class $\clF$ of $\Al$-frames is said to be {\em $f$-tunable} for a function $f:\omega\to\omega$, if
 for every $\frF\in\clF$, for every finite partition $\clV$ of $F$ there exists a refinement
$\clU$ of $\clV$ such that $|\clU|\leq f(|\clV|)$ and $\clU$ is tuned in $\frF$.
A class $\clF$ is {\em uniformly tunable}, if it is $f$-tunable for some $f:\omega\to\omega$.
\end{definition}

\improve{\ISH{History: Frantzen; Nick; ShSh}}

The following theorem is a Kripke-style version of Malcev criterion: 
\begin{theorem}\cite{LocalTab16AiML}\label{thm:LFviaTuned}
\begin{enumerate}
    \item The logic of a class  $\clF$ of $\Al$-frames is locally tabular iff $\clF$  is uniformly tunable.
    \item Every locally tabular logic is the logic of  a uniformly tunable class. 
\end{enumerate}
\end{theorem} 
 
This characterization makes many properties of locally tabular logics visible. 

\begin{example}\cite{Tacl2017}
Let $L$ be a locally tabular logic, and let $\un{L}$ 
%and $\temp{L}$  
be the expansion of $L$ with the universal modality \cite{GorankoPassy1992}. 
Then $\un{L}$ is locally tabular.  
Indeed, adding the universal modality to a pretransitive logic preserves Kripke-completeness  \cite{SpaanUniv}. Hence, $\un{L}$ is characterized by its point-generated frames, which are $L$-frames  with additional universal relation $X\times X$. Clearly, $X\times X$ is tuned with respect to any partition,  
and so this class inherits  uniform tunability from $L$-frames. 
%and $\temp{L}$ extend $L$ with canonical formulas, and hence are Kripke %complete: new axioms hold in \IS{Forgotten the argument}
\end{example}
\ISLater{Check the link to Spaan: this may be Wolter}
\ISLater{
Similar argument can be applied for enrichment of the language with the difference modality. However, 
}

\ISLater{expanded languages, admits filtration}

For a frame \(F = (X,\,(R_\Di)_{\Di\in \AlA})\)  and \(Y \subseteq X\), put \(F\restr Y = (Y,\,(R_\Di\restr Y)_{\Di\in\AlA})\), where $R_\Di\restr Y= R_\Di\cap (Y\times Y)$. 
For a class $\clF$ of  frames, put
  $$\SubFrs{\clF}=\{\frF\restr Y\mid \frF\in \clF \text{ and } \emp\neq Y\subseteq \dom{\frF}\}.$$
  %\ISH{Convention: do we exclude the empty frame? Let us do it.}

\begin{proposition}\cite{LTViaSums2022}\label{prop:LF-for-subframess}
If  $\Log\clF$ is $k$-finite for some positive $k<\omega$, then
the logic of $\SubFrs{\clF}$ is ${(k{-}1)}$-finite.
\end{proposition} 
 
A frame $F$ is a {\em cluster}, if $a R_F^* b$ for all $a,b$ in $F$. 
For a class $\clF$ of frames, $\clusters{\clF}$ is the class of cluster-frames which are restrictions on clusters (as sets) occurring in frames in $\clF$:
%$\clusters{\clF}$ denotes the class of cluster-frames occurring as subframes in frames in $\clF$:
$$
\clusters{\clF}=\{F\restr C\mid C \text{ is a cluster in } F\in\clF\}.
$$
A class $\clF$ is of {\em uniformly finite height}, if for some finite $h$, for every $F\in \clF$, we have $h(F)\leq h$.
\improve{
\ISH{Do we need it? - 
For a logic $L$, let $\clusters{L}$ be the class of clusters occurring in its frames.
}}

\smallskip
The following characterization is one of the main technical tools for our study.
\begin{theorem}\cite{LocalTab16AiML}\label{thm:supple-clusters-crit}
A logic $\Log \clF$ is locally tabular iff
$\clF$ is of uniformly finite height and $\Log\clusters{\clF}$ is locally tabular.
\end{theorem}
\ISLater{
\begin{example}
  This fact can easili  
\end{example}
}
\noindent 
{\bf Remark.}
In \cite{LocalTab16AiML}, Theorems \ref{thm:LFviaTuned} and \ref{thm:supple-clusters-crit} were stated for the unimodal case. Their polymodal generalizations are straightforward. For 
 Theorem \ref{thm:LFviaTuned},  details are given in \cite{LTViaSums2022}. 
The polymodal version of Theorem \ref{thm:supple-clusters-crit} was not considered before; for the 
 sake of rigor, we prove it in Appendix.

\ISLater{

\subsection{Reducible path property.} 

\IS{Move it here and give history}

}

\section{Criteria}\label{sec:criteria}

Throughout this section, %we assume that 
$\AlA$ and $\AlB$ are finite disjoint alphabets, 
$L_1$ is an $\AlA$-logic, $L_2$ is a $\AlB$-logic, 
$\clF$ and $\clG$ are classes of $\AlA$- and $\AlB$-frames, respectively. 

\subsection{Products of locally tabluar logics} 

In the product $L_1\times L_2$ of two Kripke complete consistent logic, local tabularity of the factors is a necessary condition for local tabularity of $L_1\times L_2$: the product is conservative over the factors in this case. 
However, it is not sufficient: 
it is known that the product of two locally tabular logics can be not locally tabular. 
The simplest example is the logic $\LS{5}^2$: that $\LS{5}$ is locally tabular is well known (for example, one can observe that in the class of frames of this Kripke complete logic every partition is tuned); at the same time $\LS{5}^2$ is not locally tabular \cite{CylindricalAlgebras1}.
We provide extra conditions which give criteria of local tabularity of the product of locally tabular logics.

The following classical fact about $\LS{5}^2$ is important for our goals.  
For a set $X$, let ${\boldsymbol X}$ be the frame $(X,X\times X)$.
For  sets  \(X,\,Y\),  the product frame \(\rect{X}{Y}\) is called a {\em rectangle}.
\begin{theorem}\label{thm:Segerberg2DimML}\cite{Segerberg2DimML}  If $\clX$ and $\clY$ are families of non-empty sets and 
    $\sup\{|X| \mid X\in\clX\}$ and $\sup\{|Y| \mid X\in\clY\}$ are infinite,  
    then $\Log\{\rect{X}{Y}\mid X\in \clX \text{ and } Y\in\clY\}$ is $\LS{5}^2$. 
\end{theorem}
\begin{remark}
The system considered in \cite{Segerberg2DimML} is a conservative extension of $\LS{5}^2$. See \cite[Page 243]{ManyDim} for other references. 
\end{remark}

\improve{
\ISH{  general fact:  
 $(X,R_1,\,R_2)$ is an \(\LS{5}^2\)-frame iff \(R_1\) and \(R_2\) are equivalence relations such that  \(R_1\circ R_2 = R_2 \circ R_1.\)  
}
}

\hide{
\((X\times Y,\,R_H,\,R_V),\) where
\begin{gather*}
  (a,\,b)R_H(a',\,b') \iff a = a';\\
  (a,\,b)R_V(a',\,b')\iff b = b'.
\end{gather*}
}

\subsection{Products of pretransitive logics}

\hide{
Fix finite disjoint $\AlA$ and $\AlB$ for alphabets. 
Let $\clF$ denote a class of $\AlA$-frames, and $\clG$ a class of $\AlB$-frames.
\ISH{Make convention and use it through the text}
}
%\todo{$m$-transitive; pretransitive}

\begin{proposition}\label{prop:corner}
Let \(R_h\) and~\(R_v\) be the unions of all horizontal and all vertical relations of \(F\times G\), respectively. Then \(R_{F\times G}^* = R_h^* \circ R_v^*=R_v^* \circ R_h^*\).
\hide{
    For any frames \(F\) and \(G,\) \(R_{F\times G}^* = R_h^* \circ R_v^*,\) where \(R_h\) and~\(R_v\) denote the unions of all horizontal and all vertical relations of \(F\times G\), respectively.
    % \ISH{Perhaps, for products it is simpler to state: Let $R_h$ be the union of all horizontal...}
    }
\end{proposition}
%\todo{  DC wording.}
\begin{proof}
If $R$ is a vertical and $S$ is a horizontal relation in $F\times G$, then 
they commute: 
$R\circ S=S\circ R$. By a routine argument, $R_v^i$ and $R_h^j$ commute for all $i,j$. It follows that $R_h^* \circ R_v^*=R_v^* \circ R_h^*$.

Since $R_h^*$ and $R_v^*$ are contained in the transitive relation    
$R_{F\times G}^*$, we have $R_h^* \circ R_v^*\subseteq R_{F\times G}^*$. 
By induction on $m$, one can easily check that $R_{F\times G}^{\leq m} \subseteq R_h^* \circ R_v^*$ for all $m$. Hence, $R_{F\times G}^* \subseteq R_h^* \circ R_v^*$. 
\end{proof}

\begin{proposition}\label{prop:product-preserve-pretrans}~
\begin{enumerate}[1.]
  \item If $F$ is $m$-transitive and $G$ is $n$-transitive, then $F\times G$ is $(m+n)$-transitive.
  \hide{
  \item If $\clF$ and $\clG$ are pretransitive, then $\clF\times \clG$ is pretransitive.
  % \ISH{Was not defined. Perhaps, we do not need it at all, just the next proposition}
  }
  \item The product of two pretransitive logics is pretransitive.
\end{enumerate}
\end{proposition}
\begin{proof}    
  The second statement is immediate from the first. To check the first statement, let \(R_h\) and~\(R_v\) be the unions of all horizontal and all vertical relations of \(F\times G\), respectively. Since $F$ is $m$-transitive, $R_h$ is $m$-transitive; similarly, $R_v$ is $n$-transitive.
  By Proposition \ref{prop:corner}, we have  $R^*_{F\times G}=R_h^* \circ R_v^*=R_h^{\leq m} \circ R_v^{\leq n}\subseteq R^{\leq m}_{F\times G}\circ R^{\leq n}_{F\times G} = R^{\leq m+n}_{F\times G}$. 
  \improve{Improve notation $R_h$.} 
\hide{
  The second statement is immediate from the first. To check the first statement, let
    \(F = \left(X,\,(R_\Di)_{\Di\in\AlA}\right),\,G = \left(Y,\, (S_\Di)_{\Di\in \AlB}\right).\) By hypothesis we have 
    \begin{equation}\label{eq:product-preserve-pretrans}
       R_F^* \subseteq R_F^{\le m}\text{ and }  S_G^*  \subseteq S_G^{\le n}. 
    \end{equation}
    Let \((a,b),\,(c,d)\in X\times Y\) and \((a,b)R_{F\times G}^*(c,d).\) \improve{Details.} 
    It is not difficult\ISH{New proposition?} to check that we have 
     \(a R_{F}^* c\) and \(b R_G^* d\).
    %Then \(a R_{F}^k c\) and \(b R_G^l d\) for some \(k ,\,l < \omega.\) 
    % \ISH{We use $a,b$ for elements of frames, it is better to use $i,j,k$ for natural}
    % Then \(a  R_F^k c\) and \(b S_G^l d.\)
    By \eqref{eq:product-preserve-pretrans}, \(a R_F^{\le m} c\) and  \(b S_G^{\le n} d.\) Then
    \[
      (a,b) R_{F\times G}^{\le m} (c,b) R_{F\times G}^{\le n} (c,d),
    \]
    so we have \((a,b) R_{F\times G}^{\le m+n} (c,d).\)  
    Therefore \(R_{F\times G}^* \subseteq R_{F\times G}^{\le m+n}.\)
    \improve{``It is not difficult to check'' should be explained in more details later. See other proof of this kind. We need to extract some general observations we regularly reuse.}
\hide{
    \item Let \(F\times G \in \clF \times \clG.\) Then \(F\) and \(G\) are pretransitive; find \(m,\,n < \omega\) such that \(F\) is \(m\)-transitive and \(G\) is \(n\)-transitive. By (i) \(F\times G\) is \((m+n)\)-transitive, hence pretransitive. Since \(F\times G \in \clF \times \clG\) was arbitrary, \(\clF \times \clG\) is pretransitive.
    \item Let \(L_1,\,L_2\) be pretransitive. Find \(m,\,n\) such that \(L_1\) is \(m\)-transitive and \(L_2\) is \(n\)-transitive. Then \(\Frames(L_1)\) is \(m\)-transitive and \(\Frames(L_2)\) is \(n\)-transitive, so \(\Frames(L_1)\times\Frames(L_2)\) is pretransitive by (ii). Find \(k\) \ISH{$k=m+n$?}such that \(\Frames(L_1)\times\Frames(L_2)\) is \(k\)-transitive. Then
    \[
      L_1\times L_2 \vd \Di_\AlA^{k+1} p \imp \Di_\AlA^{\leq k} p,
    \]
    so \(L_1\times L_2\) is \(k\)-transitive, hence pretransitive by the definition.
}%hide
}
\end{proof}

\begin{example}
$\LS{4}^2$, $\LS{4}\times \LS{5}$, $\LS{5}^2$ are 2-transitive logics. 
\end{example}

\begin{proposition}\label{prop:product-preserve-height}
If $h(F)=h_1$ and $h(G)=h_2$ for some finite $h_1, h_2$, 
then $h(F\times G)=h_1+h_2-1$.
\end{proposition}
\begin{proof} Consider  a chain \([a_0] < [a_1] < \ldots < [a_{h_1}]\) in \(\Sk F\) and a chain \([b_0]< [b_1] < \ldots < [b_{h_2}]\) in \(\Sk G.\) Then there is a 
 chain of \(h_1 + h_2 - 1\) elements in \(\Sk (F\times G)\): 
  \[
    [(a_0,\,b_0)] < [(a_1,\,b_0)] < \ldots < [(a_{h_1},\,b_0)] < [(a_{h_1},\,b_1)] < \ldots < [(a_{h_1},\,b_{h_2})]
  \]
   Hence, \(h(F\times G) \geq h_1 + h_2 -1\).

If $\Sigma$ is a chain of size $l$ in  \(\Sk (F\times G)\) with the least $[(a,b)]$ 
and the largest $[(c,d)]$, then there is a chain $\Sigma_1$ in $\Sk F$ with the least $[a]$ 
and the largest $[c]$,  and a chain $\Sigma_2$ in $\Sk G$ with the least $[b]$ and the largest $[d]$ such that $|\Sigma_1|+|\Sigma_2|=l+1$. This follows by induction on $l$.  
Hence, every chain in $\Sk (F\times G)$ has at most $h_1+h_2-1$ elements.  
\improve{induction step}
    % Let \([(a_1,b_1)]< [(a_2,b_2)]\) in \(\Sk(F\times G).\)
    % By Proposition~\ref{prop:commutativity} there exists some \((c,\,d)\in\dom F\times G\) such that \((a_1,b_1) R_h^* (c,d) R_v^* (a_2,b_2).\) Notice that it implies  \(c = a_2\) and \(d = b_1\). Then there is a chain \(\Sigma_1 \subseteq \Sk F\) from \([a_1]\) to \([a_2]\) with \(|\Sigma_1|\le h_1\), and a chain \(\Sigma_2 \subseteq \Sk G\) from \([b_1]\) to \([b_l]\) with \(|\Sigma_2| \le h_2\) Then \(\{[(a, b_1)]:\: [a]\in \Sigma_1\} \cup \{[(a_2, b)]:\: [b]\in \Sigma_2\}\) that has cardinality at most \(h_1 + h_2 -1\) (notice that \([(a_l,b_1)]\) belongs to both parts). 
%    
    % For \([(a,b)] < [(a',b')]\) in \(\Sk (F\times G)\), let us write \([(a,b)] \lhd_1 [(a',b')]\) if \(a\) and \(a'\) belong to distinct clusters in \(F,\) and \([(a,b)] \lhd_2 [(a',b')]\) if \(b\) and \(b'\) belong to distinct clusters in \(G.\) Then \([(a_i,\,b_i)] \lhd_1 [(a_{i+1},\,b_{i+1})]\) or \([(a_i,\,b_i)] \lhd_2 [(a_{i+1},\,b_{i+1})]\) for any \(i <l.\) Note that \(\lhd_1,\,\lhd_2\) are strict partial orders.
%
    % For \(j = 1,\,2\), let \(\Sigma_j\) be a maximal $\lhd_j$-chain in $\Sigma$. Then \(\Sigma = \Sigma_1 \cup \Sigma_2.\)
    % % Notice that \(\{[a_i] \in \Sk F \mid [(a_i,b_i)] \in\Sigma_1\}\) is a chain in \(\Sk F.\) Then \(|\Sigma_1| \le h(F) = h_1.\) Analogously, \(|\Sigma_2| \le h_2.\) Since \([(a_1,b_1)]\in \Sigma_1 \cap \Sigma_2,\) we have \(l \le |\Sigma_1| + |\Sigma_2| - 1 \le h_1 + h_2 - 1.\) 
    %Since \(\Sigma\) was an arbitrary chain in \(\Sk(F\times G),\) it follows that \(h(F\times G) \le h_1 + h_2 - 1.\)
\end{proof}

\begin{example}
$\LS{5}^2$ is a logic of height 1. 
\end{example}
   
For a formula \(\vf\) in the alphabet
$\{\Di_1,\Di_2\}$, we define a the translation \(\tran{\vf}\in \ML(\AlA\cup\AlB)\). Let \(\tran{\vf}\) be compatible with Boolean connectives, and let $\tran{\Di_1 \psi}= \Di_\AlA^{\leq m}  \tran{\psi}$, $\tran{\Di_2 \psi}= \Di_\AlB^{\leq n}  \tran{\psi}$. The following is straightforward:  
\begin{proposition}\label{prop:express}
 If $F$ is an $m$-transitive $\Al$-frame and  $G$ is an $n$-transitive $\AlB$-frame,
 then $F\times G\mo \tran{\vf}$ iff 
 $F^*\times G^*\mo \vf$.
 %$(\dom F, R^*_F)\times (\dom G, R^*_G)\mo \tran{\vf}$.
\end{proposition}
\hide{
\begin{proof} Straightforward.
\improve{
Immediate from \eqref{eq:pretrans}.
-- no, it does not follow from \eqref{eq:pretrans}: it should be stated for models.
}
\end{proof}
}

\begin{proposition}\label{prop:pretrans-exprS5xS5}
Let $\clF$ and $\clG$ be classes of $m$- and $n$-transitive clusters, respectively.
%Assume that $\clF$ is   $m$-transitive and $\clG$ is $n$-transitive.
If $\sup\{|F| \mid F\in\clF\}$ and $\sup\{|G | \mid G\in\clF\}$ are  infinite, then $\{\vf \in \ML(\Di_1,\Di_2) \mid \clF\times \clG\mo\tran{\vf}\}
=\LS{5}^2$.
\end{proposition}  
\begin{proof}
Follows  from Proposition
\ref{prop:express} and Theorem \ref{thm:Segerberg2DimML}.
\end{proof}

\hide{
 
If \(V\) is a set and \(f:\:X \to V\) is a function then the \emph{partition induced by \(f\)} is the quotient set \(\faktor{X}{\sim},\) where
  \[
    x\sim y \iff f(x) = f(y),\quad x,\,y\in X.
  \]

\begin{proposition}
If \(\clU\) is induced by \(f:\:X\to U\) then \(|\clU| \le |V|.\)
\end{proposition} 

  }

\smallskip

\subsection{Bounded cluster property}

\begin{proposition}\label{prop:product-with-fin}
Assume that a frame $F$ is $f$-tunable, and a frame $G$ is finite.
Then $F\times G$ is $g$-tunable for $g(n)=f\left(n^{|G|}\right)|G|$.
\end{proposition}
%\ISH{On $F$: $x\sim y$ iff $\AA i\in G$ $((x,i)\sim_\clU (y,i))$.}
\begin{proof}
  Let \(\clV\) be a finite partition of $F\times G$ , $n=|\clV|$. 

  Define  \(\alpha:\:\dom F \to \clV^{\dom G}\) by
  \[
    \alpha(a) = ([(a,\,b)]_\clV)_{b\in\dom G},\quad a\in \dom F.
  \]
  Then \(\alpha\) induces a partition of \(F\) of cardinality at most \(n^{|G|}.\) Since \(F\) is \(f\)-tunable, there is a refinement \(\clU\) of this partition such that \(|\clU|\le f\left(n^{|G|}\right)\) and $\clU$ is tuned in $F$.

 For $(a,b)$ in $F\times G$, put \(\beta(a,b) = ([a]_{\clU},\,b)\). 
 Then $\beta$ induces a refinement $\clS$ of \(\clU\) of cardinality at most \(f\left(n^{|G|}\right)| G|.\) 
 %Let $\sim_\beta$ be the corresponding equivalence on $F\times G$.
  We show that $\clS$ is tuned in $F\times G$.

  Let \(R_\Di\) be a relation of \(F,\) and suppose \((a_1,b) R_\Di^h (a_2,b)\) in \(F\times G\). Then \(a_1 R_\Di a_2.\) 
  %Let \((a_1',b_1') \sim_\beta (a_1,b).\) 
  Let \((a_1',b')\in [(a_1,b)]_\clS.\) 
  Then  $b'=b$ and 
  \([a_1']_{\clU} = [a_1]_{\clU}\).  
  We have \(a_1' R_\Di a_2'\) for some $a_2'\in [a_2]_\clU$, 
  since \(\clU\) is tuned in $F$. 
  Then \((a_1',b') R_\Di^h (a_2',b)\) and  $(a_2',b)\in [(a_2,b)]_\clS$. 

  Now let \((a,b_1) R_\Di^v (a,b_2)\), where \(R_\Di\) is a relation of \(G.\) Let \((a',b_1') \in[(a,b_1)]_\clS.\) Then \(b_1' = b_1 \), so \(b_1' R_\Di b_2.\) Moreover, $[a']_\clU$ = $[a]_\clU$, so \(\beta(a',b_2) = \beta(a,b_2)\). Hence, we have 
  $(a',b_1')  R_\Di^v (a',b_2)$ and  $(a',b_2)\in[(a,b_2)]_\clS$.
  \hide{
  \[
    (a',b_1')  = (a',b_1)\, R_\Di^v\, (a',b_2) \in[(a,b_2)]_\clS.
  \] 
  Then \((a_1',b_1')\) is related by \(R_\Di^v\) to a point from the \(\sim_\beta\)-class of \((a,b_2),\) as desired.}
\end{proof}

 \improve{What is tabular? Was it defined already?}
\begin{corollary}
If $L_1$ is locally tabular and $L_2$ is tabular, then $L_1\times L_2$ is locally tabular.    
\end{corollary}
\begin{proof}
Follows from \cite[Corllary 5.8]{MLTensor} and Proposition \ref{prop:product-with-fin}.%\IS{Simpler?}
\end{proof}

\extended{ 
\begin{proposition}\label{prop:product-with-fin-cor}
Let $\clF$ be a class of $\AlA$-frames, $\clG$ a class of $\AlB$-frames.
Assume that the logic of $\clF$ is locally tabular,
and for some finite $n$,  every $G\in \clG$ has at most $n$ elements.
Then the logic of $\clF\times \clG$ is locally tabular.
\end{proposition}
\ISH{Do we use this proposition?}
%\ISH{Locally tabular times tabular is locally tabular}
\begin{proof}
Since \(\clF\) is locally tabular, by Theorem~\ref{thm:LFviaTuned} \(\clF\) is \(f\)-tunable for some function \(f:\:\omega\to \omega.\)

Let \(F\in \clF,\,G\in\clG,\,\clU\) a finite partition of \(F\times G.\) By Proposition \ref{prop:product-with-fin}, there exists a tuned refinement \(\clU'\) of \(\clU\) with
\[
  |\clU'| \le f\left(|\clU|^{| G|}\right)\cdot| G| \le f(|\clU|^n)\cdot n.
\]
Since \(F,\,G\) were arbitrary, \(\clF\times\clG\) is \(g\)-tunable for \(g(k) = f(k^n)\cdot n.\) Therefore the logic of \(\clF\times \clG\) is locally tabular by Theorem~\ref{thm:LFviaTuned}.
\end{proof}
}

Recall that for a class $\clF$, $\clusters{\clF}$ denotes the class of clusters occurring in frames in $\clF$. 
We say that $\clF$ has the {\em bounded cluster property}, 
if for some $m<\omega$, $|C|<m$  for all $C\in \clusters{\clF}$.
\hide{
\ISH{seems redundant: 
A logic is said to have the {\em bounded cluster property}, if
the class of its frames has. }}
%\ISH{Do we use $\clusters{L}$?}
%for some $m<\omega$, $|C|<m$  for all $C\in \clusters{L}$.

\begin{theorem}\label{thm:criterion-bcp} 
Let $\clF$ and $\clG$ be non-empty classes of $\AlA$- and $\AlB$-frames, respectively. 
The logic $\Log(\clF\times \clG)$  is locally tabular iff 
the logics $\Log\clF$ and $\Log\clG$ are locally tabular, 
 and at least one of the classes $\clF$, $\clG$ has 
the bounded cluster property. 
\end{theorem}
\begin{proof}
Assume that $L=\Log(\clF\times \clG)$ is locally tabular. 
Then $L_1=\Log\clF$ and $L_2=\Log \clG $ are locally tabular, since
$L$  is a conservative extension of each of them. 
By Theorem \ref{thm:1-finite-to-m-h}, $L_1$ and $L_2$ are pretransitive.  
Let $\clC_1=\clusters{\clF}$,  $\clC_2=\clusters{\clG}$, and  $\clC=\clusters{\clF \times \clG}$. 
It follows from Proposition \ref{prop:corner} that $\clC_1\times \clC_2=\clC$. 
%By Proposition \ref{prop:LF-for-subframess}, the logic of $\clC$ is locally tabular.  
By Theorem  \ref{thm:supple-clusters-crit}, the logic of $\clC$ is locally tabular.  
On the other hand, $\LS{5}^2$ is not locally tabular. By 
Proposition \ref{prop:pretrans-exprS5xS5}, for some finite $n$ we 
have $|C|<n$ for all $C$ in $\clC_1$ or for all $C$ in $\clC_2$.  
\improve{It follows from Proposition \ref{prop:corner} - \ISH{details}}

The other direction follows from the facts that locally tabular logics have finite height (Theorem \ref{thm:1-finite-to-m-h}), which is preserved under the product (Proposition \ref{prop:product-preserve-height}). By Theorem \ref{thm:supple-clusters-crit}, it is enough to show that clusters occurring in the underlying products form a uniformly tunable class. The required partitions exist according to the bounded cluster property, in view of Proposition \ref{prop:product-with-fin}.
\end{proof}

\subsection{Reducible path property} 

For $m<\omega$, consider the following first-order property $\FORP_m$:
\begin{equation*}%\label{eq:rpp}
\AA x_0,\ldots, x_{m+1} \; (x_0Rx_1R\ldots R x_{m+1}\imp \bigvee\limits_{i<j\leq m+1} x_i = x_j
\vee \bigvee\limits_{i< j\leq m}  x_i R x_{j+1}).
\end{equation*}
We say that a class of $\clF$ of frames has the {\em reducible path property} (for short, {\em rpp}), 
if for some fixed $m$, for every $F\in \clF$, 
$\FORP_m$ holds in $(X,R_F)$.  

Notice that $\FORP_m$ is stronger than $m$-transitivity. This property is another nesessary condition 
for local tabularity:
\begin{theorem}\cite[Theorem 7.3]{LocalTab16AiML} \label{thm:ShSh:RP}
If a class of unimodal frames is uniformly tunable (equivalently, has a locally tabular logic), then it has the rpp.          
\end{theorem} 

Let $\RP_m(\Di)$ be the modal formula 
$$p_0\con \Di(p_1\con \Di(p_2\con \ldots \con \Di  p_{m+1})\ldots )\imp 
\bigvee\limits_{i<j\leq m+1} \Di^i (p_i \con p_j)
\vee \bigvee\limits_{i< j\leq m}  \Di^i(p_i \con \Di p_{j+1}).
$$

\begin{proposition}\label{prop:rpp-to-modal-uni} For a frame $F=(X,R)$, 
$F\mo \RP_m(\Di)$ iff 
$F$ satisfies $\FORP_m$.
\end{proposition}
\begin{proof}
Straightforward.\hide{\ISH{later: double check again} \ISH{Done}}
\end{proof}

Since every locally tabular logic $L$ is the logic of a uniformly tunable class, $L$ contains $\RP_m(\Di_\Al)$ for some $m$.
\hide{
%this proposition gives another necessary condition for local tabularity of a logic. 
\begin{corollary}

\end{corollary}
}

If reducible path property is sufficient for local tabularity was 
left as an open question 
(in view of Theorem \ref{thm:supple-clusters-crit}, 
it is enough to consider clusters, that is frames of height 1).
Below we give a very simple example that shows that it is not sufficient; at the same time, we show that a special form of rpp is sufficient in the case of products. 
\begin{proposition}[Negative solution to  Problem 8.1 in \cite{LocalTab16AiML}]  
There is a unimodal cluster $C$ such that $\FORP_2$ holds in $C$, while 
its logic is not locally tabular. 
\end{proposition}
\begin{proof}
Let $C=(\omega,R)$, where $m R n \tiff m\neq n+1$. 
Clearly, $C$ is a cluster: $m R^2 n$ for all $m,n$. Consider a path \(m_0 R m_1 R m_2 R m_3\) in \(C.\) If \(m_0,\,\ldots,\,m_3\) are distinct and \(m_0 = m_2 + 1\) then \(m_0 \ne m_3 +1,\) so \(m_0 R m_2\) or \(m_0 R m_3.\) In either case, the path \(m_0 R \ldots R m_3\) can be reduced. Therefore \(C\) satisfies \(\FORP_2.\)

Consider a valuation \(\theta\) on \(\omega\) with \(\theta(p) = \{0\}.\) Let \(\varphi_0 = p\),  \(\varphi_{n+1} = \lnot\Di \varphi_n\). A simple induction on \(n\) shows that \((F,\,\theta),\,n\models \varphi_m\) \tiff \(n=m.\) 
So the logic of $F$ is not    locally tabular (in fact, not 1-finite).
\hide{
Then \(\Log(F,\,\theta)\not\vd\varphi_m \lra \varphi_n\) \ISH{$\Log(F,\,\theta)$ is not a logic. But do we need it? Just: $\Log(F)$ is not 1-finite} for any \(m < n < \omega.\) Therefore \(\Log(F,\,\theta)\) contains countably many nonequivalent formulas, hence \(\Log(F,\,\theta)\) is not locally tabular. Then \(\Log(F) \subseteq \Log(F,\,\theta)\) is also not locally tabular.
% \todo{details} 
}
\end{proof}

\begin{proposition}\label{prop:rpp-implies-bcp}
Assume that $\clF^*\times \clG^*$ has the rpp. Then $\clF$ or $\clG$ has
the bounded cluster property.  
\end{proposition}
\begin{proof}
For the sake of contradiction, assume that for all $l<\omega$
there are $F\in \clF$ and $G\in\clG$, 
a cluster $C$ in $F$ and a cluster $D$ in $G$ such that $|C|,|D|>l$. 
Let $H=F^*\times G^*$. 
Pick distinct $x_0,\ldots, x_l\in C$ and distinct $y_0,\ldots,y_l\in D$
and consider $a_0,\ldots, a_{2l}$ in $H$, where 
$a_{2i}=(x_i,y_i)$, and $a_{2i+1}=(x_i,y_{i+1})$,
Then the zigzag path 
$$
a_0R_H a_1 R_H a_2 R_H \ldots R_H a_{2l}
$$ 
falsifies $\FORP_{2l-1}$ in $H$.  
It follows that $\FORP_m$ holds in $\clF\times \clG$ for no $m$, and so  $\clF^*\times \clG^*$
lacks the rpp.   
\end{proof}
\extended{\ISH{A counterexample with nooses would be informative: another non-LT logic with rpp}}

From Theorem \ref{thm:criterion-bcp} and Proposition \ref{prop:rpp-implies-bcp}, we obtain 
\begin{corollary}\label{cor:rpp-lt}
If the logics $\Log\clF$ and $\Log\clG$ are locally tabular, 
and $\clF^*\times \clG^*$ has the rpp, then 
$\Log(\clF\times \clG)$ is locally tabular.
\end{corollary}

\hide{
\ISH{\todo{This does not work! Only works for $\Di\vf =\Di_\AlA^*\vf \vee \Di_\AlB^*\vf$}}
}

\IS{Explain that in the transitive bimodal case, rpp and prpp are the same}

Let 
$\RP_m(k,n)$ denote the formula 
$\RP_m(\bar{\Di})$, where $\bar{\Di}\vf$ abbreviates $\Di_\AlA^{\leq k}\vf \vee \Di_\AlB^{\leq n}\vf$. These formulas are called {\em  product rpp formulas}.
They give another necessary axiomatic condition for local tabularity. We illustrate it for the bimodal transitive case, while the case of more modalities and pretransitive logics is a straightforward generalization. 
  For a formula \(\vf \in \ML(\Di)\), let \(t(\vf)\) be the translation of \(\vf\) that is compatible with Boolean connectives and satisfies \(t(\Di \varphi) = \Di_1 \varphi \lor \Di_2 \varphi.\)
    Let \(L_0 = \{\varphi \mid t(\varphi) \in L\}.\)
    Then \(L_0\) is a normal logic, and it is locally tabular. So it 
    contains an $\RP_m(\Di)$ formula.  Hence,  $L$ contains \(\RP_m(1,1) = t(\RP_m(\Di)) \in L.\)     So we have
\begin{proposition}\label{prop:lt_implies_rpp}
    If a bimodal logic \(L\) is locally tabular, then it contains a product rpp formula. 
\end{proposition}
\ISLater{reread this staff, and move to prel}
\hide{
\begin{proof}
    For a formula \(\vf \in \ML(\Di)\), let \(t(\vf)\) be a translation of \(\vf\) that is compatible with Boolean connectives and satisfies \(t(\Di \varphi) = \Di_1 \varphi \lor \Di_2 \varphi.\)
    Let \(L_0 = \{\varphi \mid [\varphi] \in L\}.\)
    Then \(L_0\) is locally tabular.
    By Theorem~\ref{thm:ShSh:RP}, the class of~\(L_0\)-frames satisfies \(\FORP_m\) for some~\(m < \omega.\)
    By Proposition~\ref{prop:LT_implies_fmp+extensions}, \(L_0\) is Kripke complete, so \(\RP_m(\Di)\in L_0.\)
    It follows that \(\RP_m(1,1) = [\RP_m(\Di)] \in L.\)
\end{proof}
}
\begin{proposition}\label{prop:rpp-to-modal-poly} For a frame $F=(X,(R_\Di)_{\Di\in\AlA\cup\AlB})$, we have
$F \mo \RP_m(k,n)$ iff 
$(X,(\bigcup_{\Di\in\AlA} R_\Di)^{\leq k})\cup (\bigcup_{\Di\in\AlB} R_\Di)^{\leq n})$ satisfies $\FORP_m$.
\end{proposition}
\begin{proof}
Follows from Proposition \ref{prop:rpp-to-modal-uni}.
\end{proof}

\begin{proposition}\label{prop:rpp-canon} %For any finite $\Al$, $m$, 
Formulas $\RP_m(k,n)$ are canonical.
\end{proposition}
\begin{proof}
These formulas are Sahlqvist. 
\end{proof}

\begin{proposition}\label{prop:LTimpliesrpp}
Let $\clF$ be a class of $k$-transitive $\AlA$-frames, and 
$\clG$ a class of $n$-transitive $\AlB$-frames.
If $\Log(\clF\times \clG)$ is locally tabular, 
then 
$\clF^*\times \clG^*$ has the rpp, and 
$\clF\times \clG\mo \RP_m(k,n)$ for some $m$.
\end{proposition}
\begin{proof} 
For a unimodal formula $\vf$,  we define a translation $\traone{\vf}$ by compatibility with Boolean connectives and 
$\traone{\Di \psi}=\Di_\AlA^{\leq k}\traone{\vf} \vee \Di_\AlB^{\leq n}\traone{\vf}$. 
Consider frames $F$ and $G$, and assume that $F^*\times G^*$ is $(X\times Y, R_1,R_2)$; then 
put $F \star  G=(X\times Y,R_1\cup R_2)$. 

It is straightforward that for  $k$-transitive $F$ and $n$-transitive $G$ we have 
\begin{equation}\label{eq:star-RPP}
F\times G \mo \traone{\vf} \tiff F\star G\mo \vf.    
\end{equation}
It follows that $L=\{ \vf\in \ML(\Di) \mid F\times G \mo \traone{\vf} \}$ is a unimodal logic.
Since $\Log(\clF\times \clG)$ is locally tabular, $L$ is locally tabular as well. 
Hence the class $\clH=\{F \star  G\mid F\in\clF\text{ and } G\in \clG\}$ has
the rpp \cite{LocalTab16AiML}; %this proves the first statement. 
by Proposition  \ref{prop:rpp-to-modal-poly}, $\clF\times \clG\mo \RP_m(k,n)$. 
\hide{By Proposition \ref{prop:rpp-to-modal-uni},
 $\clH\mo \RP_m(\Di)$ for some $m$. By \eqref{eq:star-RPP}, $\clF\times \clG\mo \RP_m(k,n)$. }
\end{proof}

\subsection{\(1\)-finiteness}
\improve{\ISH{This looks bad. We need to find a better wording.}}

It is known that above $\LS{4}$, 1-finiteness is sufficient for local tabularity \cite{Maksimova89}. In general, there are 1-finite logics that are not locally tabular \cite{Glivenko2021,LTViaSums2022}. We  show that for products of locally tabular logics, $1$-finiteness guarantees local tabularity. 

\extended{ 
\begin{example}
Consider the frame $(\mathbb{Z}, R)$, where $mRn$ iff $|m-n|\neq 1$. 
The logic of this frame is one-finite, but not 2-finite. 
\ISH{Write down...} 
\end{example}
}
 
\begin{theorem}\label{theorem:1finite-LT} 
%Let $\clF$ and $\clG$ be non-empty. % classes of $\AlA$- and $\AlB$-frames, respectively.
Let \(\clF,\,\clG\) be non-empty classes of frames. If the logics $\Log\clF$ and $\Log\clG$ are locally tabular, 
and $\Log(\clF\times \clG)$ is \(1\)-finite, then  $\Log(\clF\times \clG)$  is locally tabular.
\end{theorem}
\begin{proof} 
Assume that $\Log\clF$ and $\Log\clG$ are locally tabular. 
Let $\clC_1=\clusters{\clF}$, $\clC_2=\clusters{\clG}$, and let 
$\clR$ be the class of rectangles $\rect{X_1}{X_2}$ such that $X_i$ is the domain of $C_i\in \clC_i$, $i=1,2$.

For the sake of contradiction, assume that $\Log(\clF\times \clG)$ is not locally tabular. Then $\Log\clR=\LS{5}^2$
according to  Theorem \ref{thm:criterion-bcp}  and Proposition \ref{prop:pretrans-exprS5xS5}.  
It is known that $\LS{5}^2$ is not 1-finite \cite[Theorem~2.1.11(i)]{CylindricalAlgebras1}.
Hence, for each $l$, there is a model $M=(R,\theta)$ on $R\in \clR$ and 
formulas $\psi_1, \ldots, \psi_l$ in the single variable $p$ such that 
all sets $\vext(\psi_i)$ are pairwise distinct. \improve{\ISH{What is the simplest argument for this? Maltsev?}}
Consider the frames $F$ in $\clF$ and $G$ in $\clG$ with clusters $C_1$ in $F$ and $C_2$ in $G$
that form the rectangle $R$. 
Let $H$ be the subframe of \(F\times G\)
\hide{
\VS{Typo? of \(F\times G\)?}\ISH{Yes!}
}
generated by a point in $C_1\times C_2$. Consider a model $M'=(H,\theta')$ 
with $\theta'(p)=\theta(p)$. 

\improve{more details}
Assume that $l>2$. Then $\theta(p)\neq \emp$: otherwise, 
every formula in the single variable $p$ is equivalent to either $\bot$ or $\top$ on $M$.
By Theorem \ref{thm:1-finite-to-m-h}, for some $m,n$, every frame in $\clF$ is $m$-transitive and every frame in $\clG$ is $n$-transitive. 
Let $\bar{\Di} p$ abbreviate $\Di_\AlA^{\leq m} \Di_\AlB^{\leq n} p$.
By Proposition \ref{prop:corner}, we have:
$$
\text{$M',a\mo \bar{\Di} p$ iff $a$ belongs to $C_1\times C_2$.}
$$
For bimodal formulas $\vf$ in the single variable $p$, we define 
$\traone{\bot}=\bot$, $\traone{p}=p$, $\traone{\psi_1\imp \psi_2}=\bar{\Di} p\con( \traone{\psi_1}\imp \traone{\psi_2})$, and 
$\traone{\Di_1 \psi}=\bar{\Di} p \con \Di_\AlA^{\leq m}  \traone{\psi}$, 
$\traone{\Di_2 \psi}=\bar{\Di} p\con \Di_\AlB^{\leq n}  \traone{\psi}$. 
Then we have 
$$\text{
$M',a\mo \traone{\vf}$ iff  ($a$ is in $M$ and $M,a\mo \vf$).
}$$
So formulas $\traone{\psi_1}, \ldots, \traone{\psi_l}$ are pairwise non-equivalent in $M'$.
It follows that $\Log(\clF\times \clG)$ is not 1-finite.   
\end{proof}

\subsection{Criteria} 

We combine our previous observations in the following criteria. 

\begin{theorem}\label{thm:criterion-frames-general} 
Let $\clF$ and $\clG$ be non-empty.
%classes of $\AlA$- and $\AlB$-frames, respectively. 
\hide{
The logic $\Log(\clF\times \clG)$  is locally tabular iff 
the logics $\Log(\clF)$ and $\Log(\clG)$ are locally tabular, 
and one (each) of the following conditions  holds: }
TFAE: 
\begin{enumerate}
\item\label{cr-i0} $\Log(\clF\times \clG)$  is locally tabular.
\item\label{cr-i1} $\Log\clF$ and $\Log\clG$   are locally tabular and at least one of the classes $\clF$, $\clG$ has 
the bounded cluster property. 
\item\label{cr-i2} $\Log\clF$ and $\Log\clG$   are locally tabular and $\clF^*\times \clG^*$  has the rpp.
\extended{
\item\label{cr-i3} There is $N< \omega$ such that for every $H\in \clF\times \clG$, 
for every 2-element partition $\clV$ of $H$ there exists a refinement $\clU$ of $\clV$ 
such that $\clU$ is tuned in $H$ and $|\clU|\leq N$. 
}
\item\label{cr-i4} $\Log\clF$ and $\Log\clG$   are locally tabular and $\Log(\clF\times \clG)$  is 1-finite.
\end{enumerate}
\end{theorem} 
\begin{proof}
%  That \eqref{cr-i0} and \eqref{cr-i1} are equivalent was shown in Theorem \ref{thm:criterion-bcp}. 
Theorem~\ref{thm:criterion-bcp} shows the equivalence between \eqref{cr-i0} and \eqref{cr-i1}.
%Proposition \ref{prop:LTimpliesrpp} shows that \eqref{cr-i0} implies \eqref{cr-i2}. 
By Proposition \ref{prop:LTimpliesrpp},  \eqref{cr-i0} implies \eqref{cr-i2}. 
%That \eqref{cr-i2} implies \eqref{cr-i1} was shown in  Proposition \ref{prop:rpp-implies-bcp}. 
By Proposition \ref{prop:rpp-implies-bcp}, \eqref{cr-i2} implies \eqref{cr-i1}.

 By Theorem \ref{theorem:1finite-LT}, \eqref{cr-i4} implies \eqref{cr-i0}; the converse is trivial.
\end{proof}
\improve{\VS{A bit hard to follow; rewrite the repetitive syntax?}}

\hide{
\begin{remark} Perhaps, this theorem
would be more natural  in terms of classes of frames. Axiomatic version
is of separate interest.
\end{remark}
 }

\newcommand\bc{\mathrm{bc}}
\improve{DC: term} \improve{preordered set}
Let $F=(X,R)$ be a preorder  %unimodal transitive frame 
whose skeleton is converse well-founded.
%of finite height. 
Observe that $F$ has no clusters of size greater than $n$ iff $F\restr Y\toto \boldsymbol{n+1}$
for no $Y\subseteq \dom F$. 
The latter property is modally definable \cite[Theorem 9.38(ii)]{CZ}. 
%there is a modal formula $\bc_n$ such that $n$ bounds the size of clusters in $F$. 
It follows  that in the unimodal transitive frames of finite height, the bounded cluster property 
 can be expressed by modal formulas $\bc_n$, where $n$ bounds the size of clusters. \hide{
\hide{ 
It is known that in the unimodal transitive case, the bounded cluster property 
of a frame of finite height 
can be expressed by modal formulas $\bc_n$, where $n$ bounds the size of clusters.
Indeed, $F$ has no clusters of size greater than $n$ iff $F\restr Y\toto \boldsymbol{n+1}$
for no $Y\subseteq \dom F$; by \cite[Theorem 9.38(ii)]{CZ}, this property is modally expressible.
}
\hide{
It is known that in the unimodal transitive case, the bounded cluster property 
of a frame of finite height 
can be expressed by modal formulas $\bc_n$, where $n$ bounds the size of clusters.
Indeed, $F$ has no clusters of size greater than $n$ iff $F\restr Y\toto \boldsymbol{n+1}$
for no $Y\subseteq \dom F$; by \cite[Theorem 9.38(ii)]{CZ}, this property is modally expressible.
%\todo{ Check refs in \cite{Fine85,CZ} }.
}
\hide{Later this result was generalized for normal extensions of $\logics{wK4}$ \todo{ \cite{Guram-Silvio-Mamuka} }. 
\ISH{Seems we do not need it.}
} 
}Consequently, in the case of $k$-transitive $\AlA$-frames, the corresponding property
is expressed by formulas $\tranAk{\bc_n}$, which will be called 
{\em bounded cluster formulas}.
\hide{
We use the same notation $\bc_n$ for formulas of bounded clusters in this general case, assuming that the alphabet and $m$ are clear from the context. 
}
%assume that $\bc_n$ depends on the alphabet $\AlA$ and $m$. 
%$\bc_n(\AlA,m)$. \ISH{General notation}
\hide{
%in this case, substitute $\Di$ with $\Di\AlA^{\leq m}$ in $\bc_n$. 
We use the same notation $\bc_n$ for formulas of bounded clusters in this general case
  (formally, now $\bc_n$ depends on the alphabet $\AlA$ and $m$). \ISH{better wording, more details}. Hence we obtain:
  }

For a pretransitive logic $L$, let $\tra{L}$ be the least $k$ such that $L$ is $k$-transitive;
$k$ is called the {\em pretransitivity index} of $L$.
By Theorem \ref{thm:1-finite-to-m-h}, $\tra{L}$ is defined for every locally tabular logic.

\begin{corollary} \label{cor:criterion-logics-general} Let $L_1, L_2$  be Kripke complete consistent logics. 
%logics in alphabets $\AlA$ and $\AlB$.
TFAE:
\begin{enumerate}%[(i)]
\item $L_1\times L_2$  is locally tabular.
\hide{
\item $L_1$ and $L_2$ are locally tabular and at least one of them contains a bounded cluster formula.}
\item $L_1$ and $L_2$ are locally tabular and at least one of them contains a bounded cluster formula $\tranAk{\bc_n}$, where $k$ is the pretransitivity index of this logic.   
\item $L_1$ and $L_2$ are locally tabular and  $L_1\times L_2$ contains a product rpp formula 
$\RP_m(k_1,k_2)$, where $k_i=\tra{L_i}$.
\end{enumerate} 
\end{corollary}

 \section{Examples}\label{sec:examples} 

\hide{
\ISH{
``It should be easy to show that a
pretabular logic $\times$  pretabular logic is not locally tabular.'' 
This is reviewer.

It is clear the this reviewer never cared to read and understand the criterion. 

What  if we multiply all these pretabular logics for examples?  

Perhaps, we can do it in a table. Let us exclude $\Grz.3$.

}
}
 
It is well known that for unimodal logics above $\LK{4}$, local tabularity is equivalent to finite height \cite{Seg_Essay},\cite{Maks1975LT}. 
This criterion was extended for weaker systems in \cite{LocalTab16AiML}: it holds for logics containing $\Di^{k+1} p\imp \Di p\vee p$ with $k>0$. 

\begin{corollary} \label{cor:criterion-SegMaks} Let $L_1, L_2$  be Kripke complete consistent unimodal logics, and for some $k,n>0$, $L_1$  contains $\Di^{k+1} p\imp \Di p\vee p$ and 
$L_2$ contains $\Di^{n+1} p\imp \Di p\vee p$.  In this case, 
$L_1\times L_2$  is locally tabular iff 
$L_1$ and $L_2$ contain formulas of finite height and for some $m$,
$L_1\times L_2$ contains the product rpp formula $\RP_m(k,n)$. 
\hide{
for some $h,t,k$,
$L_1$ contains   $B_h^{\leq m}$, $L_2$ contains $B_t^{\leq n}$, and 
$L_1\times L_2$ contains $\RP_k(\Di)$. }
\end{corollary} %\VS{Let's discuss it.}

For a   transitive logic $L$, let  $L[h]$ be its 
extension with the axiom  $B_h$ of finite height. 

Let $\GL$ and $\Grz$  be the logics of converse well-founded  partial orders, strict and non-strict, respectively. Let 
$\GL.3$ and $\Grz.3$  be their extensions for the case of linear orders.  Clearly, these logics have the bounded cluster properties.
\ISLater{Improve wording}
%Let $\Grz.3$ be the logic of converse well-founded non-strict linear orders, 
%$\Grz.3[h]$ its extension with $B_h$.
\begin{corollary}\label{cor:grz3}
For any locally tabular  $L$, and any finite $h$, 
the  logics 
$\Grz[h]\times L$, 
$\GL[h]\times L$ are locally tabular. 
Consequently,  $L'\times L$ is locally tabular for any extension $L'$ of $\Grz[h]$ or $\GL[h]$. In particular, 
$\Grz.3[h]\times L$, 
$\GL.3[h]\times L$ are locally tabular. 
\end{corollary}

It is known \cite[Theorem 9.4]{GabbShehtProdI}\improve{DC the ref; add also ref to Lemma 9.2?}
that if an $\Al$-logic $L$ contains a formula $\Box_\Al^m\bot$ for some $m$, 
then this logic is locally tabular. 
In \cite{Shehtman2012} (see also \cite[Corollary 12.6.13]{Shehtman2018}), it was shown that $L\times L$ is locally tabular. 
We notice that the bounded cluster property holds for $L$-frames: in this case, clusters are singletons. Hence, we have the following generalization of the above result.
\begin{corollary}\label{cor:criterion-ShehtmanStyle}
If $L_1$  contains a formula $\Box_\Al^m\bot$ and $L_2$ is locally tabular, then $L_1\times L_2$ is locally tabular. 
\end{corollary}
\extended{
For a pretransitive logic $L$, let $L[h]$ be the extension of $L$ with the finite height formula 
$\tranA{B_h}$, where $m=\tra{L}$. \ISH{To prel?}
}
\improve{
\ISH{Shehtman, 2012: Squares of modal logics with additional connectives - product of such logics}
Or Quote \cite{Shehtman2018}? I could not fine LF product in 2012, but in 2018 there is a reference to it.
- It follows from lemma 6.10.}

\section{Product finite model property %of locally tabular logics
}\label{sec:pfmp}
\ISLater{Refs to: Shehtman; ReynoldsZakh}
 A modal logic \(L\) has the \emph{product fmp}, if \(L\) is the logic of a class of finite product frames.
 %For product logics, 
 The product fmp is stronger than the fmp: for example, \(\LK4\times \LS5\) has the fmp~\cite[Theorem~12.12]{GabbayShehtman-ProductsPartI}, but lacks the product fmp~\cite[Theorem~5.32]{ManyDim}. 
Examples of logics with this property are also known: they include 
$\LK{}\times \LK{}$ (where $\LK{}$ stands for the smallest unimodal logic) and $\LK{}\times\LS{5}$ \cite{ManyDim}, \ISLater{Improve the ref}
%$\LS{4}\times \LS{5}$ \cite{GabbayShehtman-ProductsPartI}
the logic \(\LS5\times \LS{5}\)  \cite{Segerberg2DimML}, or  
its extensions \cite{NickS5}.\ISLater{Kravtsov?}\ISLater{Double check}
 \hide{
 The product fmp holds for some locally tabular logics, such as the  logics above~\(\LS5^2\)~\cite{NickS5}.\IS{Double check} The logic \(\LS5^2\) itself has product fmp~\cite{Segerberg2DimML}, but lacks local tabularity~\cite{CylindricalAlgebras1}.
 }

The product of a logic possesing the fmp with a tabular logic 
has the product fmp (\cite{MLTensor}, see Proposition \ref{prop:tabular_product_fmp}  below).  In this section, we show that the weaker property of local tabularity is not sufficient for the product fmp. 

\subsection{Product fmp fails for a locally tabular product logic}
We define the \emph{saw frame} \(F_S = (W,S,S_l,S_r)\) for the alphabet~\(\Al = \{\Di,\,\Di_l,\,\Di_r\}\) as follows:
\begin{gather*}
    W = \{u\} \sqcup \{v_i\}_{i < \omega} \sqcup \{w_i\}_{i < \omega};
    \\
    S = \{u\} \times  \{v_i\}_{i < \omega};
    \\
    S_l = \{(v_i,w_i)\}_{i < \omega};
    \\
    S_r = \{(v_i,w_{i+1})\}_{i < \omega}.
\end{gather*}
The saw frame is shown in Figure~\ref{fig:saw}.
\begin{figure}[tb]
    \centering
    \begin{tikzcd}
        w_0 & w_1 & w_2 & \cdots & w_n & w_{n+1} & \cdots\\
        v_0 \ar[u,"l"] \ar[ur,"r"] & v_1 \ar[u,"l"] \ar[ur,"r"] & v_2 \ar[u,"l"] \ar[ur,"r"]& \cdots \ar[ur,"r"] & v_n \ar[u,"l"] \ar[ur,"r"] & v_{n+1} \ar[u,"l"] \ar[ur,"r"]& \cdots\\
        &&&u \ar[ulll]\ar[ull]\ar[ul]\ar[ur]\ar[urr]
    \end{tikzcd}
    \caption{The saw frame \(F_S\).}
    \label{fig:saw}
\end{figure}
Let~\(\LSaw = \Log(F_S).\) The next proposition follows from simple semantic observations:
\begin{proposition}
  The following formulas are theorems of~\(\LSaw\):
  \begin{gather}
    \Box_\Al^3 \bot; \label{eq:saw_deadend}
    \\
    \Box_l \Box_\AlA \bot; \label{eq:saw_l_deadend}
    \\
    \Box_r \Box_\AlA \bot; \label{eq:saw_r_deadend}
    \\
    \Di_l p \to \Box_l p; \label{eq:saw_l_functional}
    \\
    \Di_r p \to \Box_r p; \label{eq:saw_r_functional}
    \\
    \Di \Di_l p \to \Di \Di_r p; \label{eq:saw_expand}
  \end{gather}
\end{proposition}

\begin{proposition}\label{prop:saw_x_s5_lt}
    \(\LSaw\times \LS5\) is locally tabular.
\end{proposition}
\begin{proof}
    Follows from Corollary~\ref{cor:criterion-ShehtmanStyle} by \eqref{eq:saw_deadend}.
\end{proof}

We define the formula~\(\varphi\in \ML(\Al\cup \{\Di_{\LS5}\})\) as~\(\varphi = \varphi_1 \land \varphi_2 \land \varphi_3,\) where
\begin{gather*}
    \varphi_1 = \Di \top;
    \\
    \varphi_2 = \Box \Di_{\LS5} \left(\Di_l \lnot p \land \Di_r p\right);
    \\
    \varphi_3 = \Box \Box_{\LS5} \left(\Di_l p \to \Di_r p\right).
\end{gather*}
\begin{proposition}\label{prop:saw_phi_consistent}
  The formula \(\varphi\) is consistent with~\(\LSaw \times  \LS5.\)
\end{proposition}
\begin{proof}
  It suffices to show that~\(\varphi\) is satisfiable in~\(F_S\times \bomega.\)
  Let~\(\theta\) be a valuation in~\(F_S\times \bomega\) given by
  \(\theta(p) = \{(w_i,j) \mid i < \omega,\,j < i\}\).
  Then a direct evaluation shows that~\(F_S\times \bomega,\theta,(u,0) \models \varphi.\)
\end{proof}

\begin{proposition}\label{prop:saw_phi_infinite}
  If~\(\varphi\) is satisfiable in a~\(\LSaw\times \LS5\)-frame \(F\times  G\), then~\(F\times  G\) is infinite.
\end{proposition}
\begin{proof}
  Let~\(\LSaw\) be valid in a rooted frame~\(F = (X,R,R_l,R_r)\) with a root~\(r\) and let~\(G = (Y,Y\times  Y).\)
  Assuming that~\(\varphi\) is true at~\(F\times G,\theta,(r,s)\) for some valuation~\(\theta\) and a point~\(s \in Y,\) we will show that \(X\) is infinite.

  We define the mapping~\(V:\:X \to Y\) by
  \(V(a) = \{b\in Y \mid (a,b)\in \theta(p)\}.\)
  Let us show by induction that there exist countable subsets~\(\{m_i\}_{i < \omega}\) and~\(\{t_i\}_{i < \omega}\) of~\(X\) such that the following are true for all~\(i < \omega\):
  \begin{gather}
    r R m_i;
    \\
    R_l(m_i) = \{t_i\}; \label{eq:mi_left}
    \\
    R_r(m_i) = \{t_{i+1}\}; \label{eq:mi_right}
    \\
    R_F(t_i) = \varnothing; \label{eq:ti_deadend}
    \\
    V(t_i) \subsetneq V(t_{i+1}) \label{eq:volume_increases};
    \\
    m_i \not\in \{m_j\}_{j < i}.
  \end{gather}

  For the base case, observe that~\(\varphi_1\) is true at~\((r,s),\) so \((r,s) R^h (m_0,s)\) for some~\(m_0 \in X.\)
  Then~\(r R m_0.\)

  For the transition, we assume that~\(\{m_j\}_{j \le i}\) and~\(\{t_j\}_{j < i}\) are already constructed.
  % By \eqref{eq:saw_expand} we have~\(R\circ R_l \subseteq R\circ R_r,\) so~\(a R m_{i-1} R_r t_i\) implies that~\(a R m_i R_l t_i\) for some~\(m_i\in X.\)
  Then \((r,s) R^h (m_i,s)\), so~\(\Di_{\LS5} \left(\Di_l \lnot p \land \Di_r p\right)\) is true at~\((m_i,s),\) so there exists~\(b_i\in Y\) such that~\((m_i,b_i) R_l^h (t_i,b_i)\) and~\((m_i,b_i) R_r^h (t_{i+1},b_i)\) for some~\(t_i\) and~\(t_{i+1}\) in~\(X\) such that~\((t_i,b_i)\not\in \theta(p)\) and~\((t_{i+1},b_i)\in \theta(p).\)
  By~\eqref{eq:saw_l_functional} and~\eqref{eq:saw_r_functional},~\(R_l(m_i) = \{t_i\}\) and~\(R_r(m_i) = \{t_{i+1}\}\).
  Since~\(\varphi_3\) is true at the root,~\(V(t_i) \subseteq V(t_{i+1}).\) The inclusion is strict because~\(b_i \in V(t_{i+1})\setminus V(t_i).\)
  % It follows that~\(t_i \ne t_{i+1}.\)
  By~\eqref{eq:saw_l_deadend} and \eqref{eq:saw_r_deadend} we also have~\(R_F(t_i) = R_F(t_{i+1}) = \varnothing.\)
  %thus~\(t_i,\,t_{i+1} \not\in \{m_j\}_{j < i}\).
  Notice that~\(r R m_i R_r t_{i+1},\) so by~\eqref{eq:saw_expand} there exists~\(m_{i+1}\in X\) such that~\(r R m_{i+1} R_l t_{i+1}.\)
  By~\eqref{eq:volume_increases}, \(t_{i+1} \ne t_j\) for all \(j \le i.\)
  Then by \eqref{eq:mi_left} \(t_{i+1}\not\in R_l(m_j)\), so \(m_{i+1} \ne m_j\) for \(j \le i.\) 
  % Since~\(\varphi_2\) is true at~\((r,s)\), we have~\(R^h_F[R^v(m_{i+1})] \ne \varnothing,\) thus by commutativity \(R^h_F(m_{i+1}) \ne \varnothing.\)
  % Then by~\eqref{eq:ti_deadend}~\(m_{i+1}\not\in \{t_j\}_{j\le i}.\)
  The induction is complete.

  It follows from \eqref{eq:volume_increases} that \(t_i\) are pairwise distinct for all~\(i < \omega.\)
  Then~\(X\) contains an infinite subset \(\{t_i\}_{i < \omega}\), and therefore \(X\) is infinite.
\end{proof}
\begin{theorem}\label{prop:saw_logic_no_pfmp}
  The local tabularity of a product logic does not imply the product fmp.
\end{theorem}
\begin{proof}
  Follows from Proposition~\ref{prop:saw_x_s5_lt} and Proposition~\ref{prop:saw_phi_infinite}.
\end{proof}

\subsection{Product fmp and products of small height}
As follows from Theorem~\ref{thm:1-finite-to-m-h} and Corollary~\ref{cor:criterion-logics-general}, it is necessary for any locally tabular product of two logics that both have a finite height and one of them contains a bounded cluster formula.
We can classify the locally tabular logics~\(L_1 \times L_2\) with respect to the height of~\(L_1\) and~\(L_2\).
For this section, we will assume that \(L_1 \vdash [bc_n]^{\tra{L_1}}\) for some \(n < \omega\).
We will describe how the product fmp depends on the height of~\(L_1\) and~\(L_2.\)

By Proposition~\ref{prop:saw_logic_no_pfmp}, the product fmp fails even when~\(h(L_2) = 1\).
Thus we will focus on the height of~\(L_1.\)

If~\(L_1\) has unit height, then the product fmp follows from the following result. A logic is called \emph{tabular}, 
if it is the logic of a single finite frame.

\begin{proposition}\cite[Corollary~5.9]{MLTensor}\label{prop:tabular_product_fmp}
    If \(L\) is a tabular logic and \(L'\) has the fmp, then \(L\times L'\) has the product fmp.
\end{proposition}
    Observe that \(L_1\) is Kripke complete by~\ref{prop:LT_implies_fmp+extensions}.
     Since \(L_1\) is locally tabular, it is the logic of its class of frames. 
    Every \(L_1\)-frame has unit height, so it consists of one cluster, which has size at most~\(n\) by the frame condition of~\([bc_n]^{\tra{L}}\).
    Then \(L_1\) is tabular by Proposition~\ref{prop:tabular_product_fmp}.
\begin{corollary}
    If \(h(L_1) = 1,\)  then \(L_1 \times L_2\) has the product fmp.
\end{corollary}

Proposition~\ref{prop:saw_logic_no_pfmp} shows that the product fmp can fail when~\(h(L_1)=3.\)
The tabularity of~\(L_1\) or~\(L_2\) is a sufficient condition for the product fmp in this case, as well as for greater~\(h(L_1)\).
Describing weaker sufficient conditions is an open problem.

Finally, the case when~\(L_1\) has height~\(2\) is the most intriguing.
Neither a proof, nor a counterexample for the product fmp of such logics is known yet. 

%~\\
\pagebreak

\section{Logics above $\LS{4}\times\LS{5}$}\label{sec:ProdaboveS5}

\ISLater{Intro words}

\ISLater{ 
\ISH{Some words; simplify notation} 

In this section, all logics are assumed to be normal extensions of the logic $\LS{4}\times \LS{4}$. 

\IS{Simplify notation:
$\RP_m$ stands for $\RP_m(1,1)$. 

Check the font  for $\RP$.

}

\IS{Products with $\LS{5}$; \\
Products with pretabular logics;\\
Prelocal  tabularity
}

\IS{Two goals: }
\begin{enumerate}
    \item There are infinitely many prelocally tabular logics in $\LS{4}\times \LS{4}$.
    \item 
Every non-locally tabular logic logic is contained in a  pre-locally tabular logics. 
We notice that even in the unimodal case, it is unknown if 
every non-locally tabular logic is contained in a prelocally tabular.  \cite[Problem 12.1]{CZ}
\end{enumerate}

\IS{Methods:
Productivization;   (product)RPP for the old task on Kripke-completeness.  

}

\newpage 
}

\ISLater{

\begin{lemma}[No proof is known yet]
Let  $L$ contains $\LS{4}[h_1]\times \LS{4}[h_2]$.  If every Kripke complete extension of  $L$ 
is locally tabular, then every its extension is locally tabular. 
\end{lemma}

\begin{theorem}[No proof is known yet]
    Every non-locally tabular extension of $\LS{4}\times \LS{4}$ 
    is contained in a locally tabular logic. 
\end{theorem}
}

\ISLater{

\begin{lemma}[No proof is known yet]
Let  $L$ contains $\LS{4}[h]\times \LS{5}$.  If every Kripke complete extension of  $L$ 
is locally tabular, then every its extension is locally tabular. 
\end{lemma}

\begin{theorem}\label{thm:finalDream}[No proof is known yet]
    Every non-locally tabular extension of $\LS{4}[h]\times \LS{5}$ 
    is contained in a locally tabular logic.
\end{theorem}

Steps: 
\begin{enumerate}
    \item $L$  is  2-transitive by Proposition \ref{prop:product-preserve-pretrans}, and the corresponding  modality is given by $\Di_1\Di_2$ (which is equivalent to $\Di_2\Di_1$).
    
    \item Assume $L$  is not locally  tabular. Trivially, $L$  is consistent. Hence, $L[1]$ is consistent as well \cite[Theorem ...]{Glivenko2021}. 

    \item 
   Consider the least natural $m$ such that $L$ is not $m$-finite.

\end{enumerate}

%\newpage 

}

\subsection{The product rpp criterion}
%\subsection{Products with one-dimensional tack and its relatives}
%\subsection{One-dimensional tack and products}
\ISLater{I suggest for this section: One-dimensional tack and products; then merge your section with this one.
And for the next section: products and two-dimensional tack}

Recall that a logic is {\em tabular}, if it is characterized by a single finite frame. A 
non-tabular logic is {\em pretabular}, if all of its proper extensions are tabular. 
There are exactly five pretabular logics above $\LS{4}$, see \cite{EsakiaMeskhi1977}, \cite{MaksimovaPretab75}, or \cite[Section 12.2]{CZ}. One of them is the logic $\LS{5}$. 
Another is the logic $\TL$ defined below. 

An important fact about $\LS{5}^2$ was established in \cite{NickS5}: this logic is {\em prelocally tabular}, that is, the following theorem holds. 
\begin{theorem}\cite{NickS5}\label{thm:S5Nick}
All proper extensions of $\LS{5}^2$ are locally tabular.    
\end{theorem} 

While $\LS{5}^2$ is prelocally tabular,
it turns out that products of pretabular logics can be not prelocally tabular even for the case of height 2. 

The {\em (one-dimensional) tack frame} $\TF$ is the ordered sum $(\omega+1,R)$ of a countable cluster and a singleton, that is
$\alpha R \beta$ iff $\alpha<\omega$ or $\beta=\omega$. 
Let $\TL$ be the logic of its frame. 
$\TL$ is one of exactly five pretabular
%(that is, characterized by a single finite frame) 
logics above $\LS{4}$.

In this subsection we show that $\TL\times \LS{5}$ is not prelocally tabular, and give an axiomatic criterion of local tabularity 
for a family of logics that contains \(\TL\times \LS5\) and its extensions.

\begin{proposition}
$\TL\times \LS{5}$ is not prelocally tabular. 
\end{proposition}
\begin{proof}
%Consider the frame $H=(2,\leq)\times\mathbf{2}$. 
%Readily, 
%$H$ validates $\TACK\times \LS{5}$. 
Let $H=T\times \boldsymbol{\omega}$. 
The formula 
$\vf=\Di_2\Di_1\Box_1 p\imp \Box_1\Di_1 p$ is not valid in $H$: consider a model
$M=(H,\theta)$ with $\theta(p)=\{(\omega,1)\}$; then $\vf$ is falsified in $M$ at $(0,0)$.

Consider the equivalence $\equiv$ on $\dom H$ that extends the diagonal 
by the set of pairs of form $((\omega,m),(\omega,n))$.\improve{ \ISH{Clumsy}}
It is easy to check that $H  \toto \fact{H}{\equiv}$, 
hence the logic $L$ of the latter frame extends $\TL$. 
It is also straightforward that $\vf$ is valid in $\fact{H}{\equiv}$, so $L$ is a proper extension of $\TL$. Finally, we observe that $L$ is not locally tabular:  
$\fact{H}{\equiv}$ contains $\rect{\omega}{\omega}$ as a subframe, and the logic $\LS{5}^2$ of this subframe is not locally tabular. 
\end{proof}
\ISLater{DC "easy to check}

We say that the \emph{product rpp criterion} holds for a logic \(L\), if for any normal extension \(L'\) of \(L\), we have:
%it is true that
\begin{center}
    $L'$ contains a product rpp formula iff  $L'$ is locally tabular. 
\end{center}

\smallskip 

  Recall that $L[h]$ denotes the extension of $L$ with the axiom  $B_h$ of 
  finite height.

  \smallskip

The unimodal logic $\LS{4.1}$ is the extension of $\LS{4}$ with the axiom 
$\Box\Di p\imp \Di \Box p$. 
In transitive frames,  $\Box\Di p\imp \Di \Box p$ defines the first-order property
\begin{equation}\label{eq:s4.1-condition}
  \forall x\,\exists y\,\left(xRy \land \forall z \left(yRz \to z = y\right)\right),
\end{equation}
see, e.g., \cite[Section 3.7]{BDV}.

Hence, in our notation, the unimodal logic $\LS{4.1}[2]$ is the extension of $\LS{4}$ with the axioms $\Box\Di p\imp \Di \Box p$ and $B_2$.  
Now we are going to show that the product rpp criterion holds for \(\LS{4.1}[2] \times \LS5\).

\begin{lemma}\label{lem:opposite_arrows}
  Let~\(F=(X,R,S)\) validate a product $L_1 \times L_2$ of transitive logics,  where~\(L_1\) has finite height.
  Then~\(R\cap S\inv \subseteq R\inv.\)  
\end{lemma}
\hide{
            \VS{We tried to weaken the assumptions. It didn't work out. Here is what we tried.}
            \begin{lemma}\label{lem:opposite_arrows}
              Let~\(F=(X,R,S)\) be a bimodal frame where both relations are transitive, \((X,R)\) has a finite height and \(R\circ S = S \circ R\).
              Then~\(R\cap S\inv \subseteq R\inv.\) 
            \end{lemma}
            \begin{remark}
                It is not assumed that $F=(X,R,S)$  is a product frame. 
            \end{remark}
            \begin{proof}
                Let \(h\) denote the height of \(X,R)\)
                Let \(u\) and \(v\) be points in \(X\) such that \(u (R\cap S\inv) v.\)
                Then \(u (R \circ S) u\), hence \(u (R\circ S)^n u\) for all \(n < \omega\).
                Since \(R\circ S = S \circ R,\) we have \(u (R^n\circ S^n) u\) for all \(n < \omega\), so there exists a sequence \(\{u_i\}_{i \le h}\) such that \(u_0 = u,\) \(u_h S^{n+1} u\) and \(u_i R u_{i+1}\) for all \(i < \omega.\)
                But \((X,R)\) is transitive and has height \(h\), so \(u_{i+1} R u_i\) for some \(i \le h.=\)
                
                !!! NOT FINISHED YET
            
                %Let \(\tra{L_1\times L_2} = m.\)
                % Let \(\Di \varphi\) abbreviate  \(\Di_1 \varphi \lor \Di_2 \varphi,\)
                %We define
                %\[\varphi = p \land \Box^{\leq m} (p\to \Di_1 q ) \land \Box^{\leq m}(q \to \Di_2 p) \to \Di^{\le m} (q \land \Di_1 p).\]
                % and let $\vf$ be 
                % \[p \land \Box^{\leq 2} (p\to \Di_1 q ) \land \Box^{\leq 2}(q \to \Di_2 p) \to \Di^{\le 2} (q \land \Di_1 p).\]
                
                % We claim that \(\varphi\in L_1\times L_2.\)
                % Indeed, let the antecedent of \(\varphi\) be satisfied at a point\(u_0\)  under some valuation \(\theta\) in \(F.\)
                % Observe that \((F,R\cup S)\) is \(2\)-transitive, so \(p\to \Di_1 q\) and \(q\to \Di_2 p\) are true at any point in \(R_{F}^*(u)\).
                % Consequently, there exist sequence \(\{u_i\}_{i < \omega}\) and \(\{v_i\}_{i < \omega}\) of points in \(F\) such that \(u_i \in \theta(p),\) \(v_i \in \theta(q)\) and \(u_i R v_i S u_{i+1}\) for all \(i < \omega.\)
            \end{proof}
}
\begin{proof}
    %Let \(\tra{L_1\times L_2} = m.\)
    Let \(\Di \varphi\) abbreviate  \(\Di_1 \varphi \lor \Di_2 \varphi,\)
    %We define
    %\[\varphi = p \land \Box^{\leq m} (p\to \Di_1 q ) \land \Box^{\leq m}(q \to \Di_2 p) \to \Di^{\le m} (q \land \Di_1 p).\]
    and let $\vf$ be 
    \[p \land \Box^{\leq 2} (p\to \Di_1 q ) \land \Box^{\leq 2}(q \to \Di_2 p) \to \Di^{\le 2} (q \land \Di_1 p).\]
    
    We claim that \(\varphi\in L_1\times L_2.\)
    Indeed, let the antecedent of \(\varphi\) be satisfied at a point\((a_0,b_0)\) of a model~\((G\times H,\theta),\) where \(G\models L_1\) and~\(H\models L_2.\)
    Then \(G\times H\) is \(m\)-transitive, so \(p\to \Di_1 q\) and \(q\to \Di_2 p\) are true at any point in \(R_{G\times H}^*(a_0,b_0).\) 
    Consequently, there exists a sequence \(\{(a_i,b_i)\}_{i < \omega}\) of points in \(G\times H\) such that \((a_i,b_i) \in \theta(p),\) \((a_{i+1},b_i) \in \theta(q)\) and \((a_i,b_i) R^h (a_{i+1},b_i) R^v (a_{i+1},b_{i+1})\) for all \(i < \omega.\)
    Then \(a_i R_G a_{i+1}\) for all \(i.\)
    Since \(L_1\) has finite height, \(a_{i+1} R_G a_i\) for some \(i < \omega.\)
    Then \((a_{i+1},b_i) R^h (a_i,b_i),\) thus \(q \land \Di_1 p\) is true at \((a_{i+1},b_i).\) Then the consequent of \(\varphi\) is true at \((a_0,b_0).\)

    Now let \(a R b\) and \(b S a\) for some \(a\) and \(b\) in \(F.\) Let \(\theta\) be a valuation on \(F\) such that \(\theta(p) = \{a\}\) and \(\theta(q) = \{b\}.\)
    Then \((F,\theta),a\models \varphi\) since \(\varphi \in L_1 \times L_2\), and therefore \(b R a.\)
\end{proof}

\begin{remark}
    The formula \(\vf\) from the proof of this lemma shows that the pairs of logics \(L_1,\,L_2\), where both \(L_1\) and \(L_2\) contain \(\LS4\), \(1 < h(L_1) < \omega\) and \(p \to \Box p \not \in L_2\), are not \emph{product-matching} in the sense of \cite[Section~5.1]{ManyDim}. 
    Indeed, consider \(F_1 = (\{0,1\},\le,\ge)\) and \(F_2 = (\{0,1\},\le,\{0,1\}^2).\) Then \(\vf\) is refuted in both \(F_1\) and \(F_2\). This observation extends the one of \cite[Theorem 8.2]{GabbShehtProdI} (or cf. \cite[Theorem~5.17]{ManyDim}), where a similar reasoning describes the case where \(\Grz \subseteq L_1\).
\end{remark}

    %\IS{Cf. \cite[p 235]{ManyDim}, formula 5.9 in the theorem before.  Give ref to this fact and most probably to \cite{GabbShehtProdI}}.

A point \(a\in X\) is called \emph{\(R\)-terminal}
 in a transitive frame \((X,R)\), if \(R(a) = \{a\}\).
 
\begin{lemma}\label{lem:S4.1[h]xS5_terminal}
    Let \(F = (X, R, \approx)\models{\LS{4.1}[h]\times \LS5}\). If \(a\in X\) is \(R\)-terminal, then \([a]_\approx\) contains only \(R\)-terminal elements.
\end{lemma}
\begin{proof}
    Let \(a\) be \(R\)-terminal and consider \(b\in [a]_\approx.\)
    By the frame condition of~\(\LS{4.1}[h],\) there exists an \(R\)-terminal element \(c\in R(b)\).
    Then \(a \approx b R c\), so by commutativity~\(a R d \approx c\) for some \(d\in X.\)
    %But since \(a\) is \(R\)-terminal, \(a \approx c.\)
    \ISLater{Modified: }
    But since \(a\) is \(R\)-terminal, $a=d$, and so \(a \approx c.\)
    Hence, $b\approx c$.
    Then \((b,c)\in R \cap {\approx},\) so \(c R b\) by Lemma~\ref{lem:opposite_arrows}.
    Finally, \(b = c\) because \(c\) is \(R\)-terminal.
\end{proof}

 Let~\(\clF_h\) be the class of all \(\LS{4.1}[h]\times \LS5\)-frames that do not validate \(\LS5^2.\)
  For any~\(F = (X,R,\approx)\in \clF_h,\) 
  %let~\(Z_F \subseteq X\) 
  let~\(Z_F\) 
  be the set of all \(R\)-terminal elements in \(F.\)
  We define \(\alpha(F)\) to be the subframe \(F \restr\left(X \setminus Z_F\right).\)  Notice that if \(Z_F = X,\) then \(R\) is an equivalence relation, so \(F \models \LS5^2.\) It follows that \(\alpha(F)\) is not empty for any \(F\in \clF_h.\) \ISLater{non-empty?}

  Let \(L_1\) be an \(\AlA\)-logic and \(L_2\) be an \(\AlB\)-logic.
  The \emph{fusion} \(\fusion{L_1}{L_2}\) is the smallest \((\AlA\cup \AlB)\)-logic that contains \(L_1 \cup L_2.\)
  The \emph{commutator} \(\LCom{L_1}{L_2}\) is the smallest logic that contains \(\fusion{L_1}{L_2}\) and the axioms \(\com(a,b),\,\com(b,a)\) and \(\chr(a,b)\) for all \(a\in \AlA\) and \(b\in\AlB,\) where
  \begin{align}
      \com(a,b) &= \Di_a \Di_b p \to \Di_b \Di_a p;
      \\
      \chr(a,b) &= \Di_a \Box_b p \to \Box_b \Di_a p.
  \end{align}
  The validity of \(\com(a,b)\) in an \((\AlA\cup\AlB)\)-frame 
  %\((X,(R_a)_{a\in \AlA}\cup (R_b)_{a\in \AlB})\) 
  \((X,(R_a)_{a\in \AlA}, (R_b)_{a\in \AlB})\) is equivalent to the \emph{commutativity} \(R_a\circ R_b = R_b \circ R_a.\)
  The formula \(\chr(a,b)\) defines the \emph{Church-Rosser property}
  \begin{equation}
    \forall x\,\forall y\,\forall z\,(x R_a y\,\&\,x R_b z \to \exists u\,(y R_b u\,\&\,z R_a u))
  \end{equation} 
  For any \(L_1\) and \(L_2,\) the product logic \(L_1\times L_2\) is a normal extension of~\(\LCom{L_1}{L_2}\) \cite[Section~5.1]{ManyDim}.

  \begin{proposition}
      For any \(F\in \clF_h,\) where \(h < \omega,\) the logic \(\LCom{\LS{4}[h-1]}{\LS5}\) is valid in \(\alpha(F)\).
  \end{proposition}
  \begin{proof}
      The validity of \(\LS{4}\) and \(\LS{5}\) is defined by universal sentences, so \(\alpha(F)\) validates the fusion~\(\fusion{\LS{4}}{\LS{5}}\),
      since \(\alpha(F)\) is a subframe of~\(F.\)

    Let~\(S\) be an~\(R\)-chain in~\(\alpha(F).\)
    Then \(S\) is also a chain in~\(F.\)
    Observe that \(F\models B_h,\) then \(F\) contains~\(R\)-chains of size at most~\(h.\)
    We claim that~\(|S|<h.\)
    Since~\(F\models \LS{4.1},\) any chain of size~\(h\) in~\(F\) contains an \(R\)-terminal element, which belongs to~\(Z_F\).
    But \(S \subseteq W \setminus Z_F,\) so~\(|S| < h.\)
    Since~\(S\) was arbitrary,~\(\alpha(F) \models B_{h-1}.\)

    Let us show the commutativity.
    Let~\(a R b \approx c\) in~\(\alpha(F).\)
    Then the same holds in~\(F\), so there exists~\(d\) in~\(F\) such that~\(a \approx d R c.\)
    Then \(d\) is not \(R\)-terminal since \(d R c\) and \(c\not \in Z_F.\)
    Then \(d\) belongs to~\(\alpha(F).\)
    
    Let~\(a \approx b R c\) in~\(\alpha(F).\)
    There exists~\(d\) in~\(F\) such that~\(a R d \approx c.\)
    By Lemma~\ref{lem:S4.1[h]xS5_terminal}, \(d\not\in Z_F\), so \(d\) is in \(\alpha(F).\)

    Finally, we show the Church-Rosser property.
    If~\(a R b\) and~\(a \approx c\) in~\(\alpha(F),\) then~\(b \approx d\) and~\(c R d\) for some \(d\) in~\(F.\)
    Then~\(d \not\in Z_F\) since~\(b\not\in Z_F\) while~\(b \in [d]_\approx.\)
    Thus~\(d\) belongs to~\(\alpha(F)\) by Lemma \ref{lem:S4.1[h]xS5_terminal}.
  \end{proof}

  \begin{proposition}\label{prop:alpha_transfers_rpp}
     Let \(h < \omega.\)
     If \(F \subseteq \clF_h,\) then \(F\) validates a product rpp formula iff~\(\alpha(F)\) does.
  \end{proposition}
  \begin{proof}    
    For the 'if' direction, observe that for any \(a,\,b,\,c\in Z_F\) that form a path we have \(a \approx c.\) 
    Assume \(\alpha(F)\models \RP_m(1,1).\) It is straightforward that \(F\models \RP_{m+2}(1,1)\).% whenever \(\alpha(F)\models \RP_m(1,1).\)\ISLater{Say a couple more words}

    The 'only if' direction is true since \(\FORP_n\) is a universal sentence for any \(n.\)
  \end{proof}

  \begin{proposition}\label{prop:alpha_transfers_lt}
    For any class of frames \(\clF \subseteq \clF_h,\) \(\Log \clF\) is locally tabular iff~\(\Log \alpha[\clF]\) is locally tabular.
  \end{proposition}
  \begin{proof} 
The `only if' direction follows from Proposition \ref{prop:LF-for-subframess}.

Assume that \(\Log \alpha[\clF]\) is locally tabular.
The class $\clF$ has uniformly finite height. By Theorem \ref{thm:supple-clusters-crit},  we only need to show that~\(\clusters{\clF}\) is uniformly tunable.
%    Lemma~\ref{lem:S4.1[h]xS5_terminal} implies that 
    %\(\clusters{\alpha[\clF]} \subseteq \clusters{\clF}\). Moreover, 
    If~\(C\in\clusters{\clF}\), then either~\(C\in\clusters{\alpha[\clF]}\) or  
    \(C\cong \rect{1}{S}\) for a set $S$. 
    Trivially, the class of rectangle frames of the form \(\rect{1}{S}\)  is uniformly tunable. 
    Hence, $\clusters{\clF}$ is the union of two uniformly tunable classes, and hence is uniformly tunable as well.\footnote{This is a particular case of the following general fact. 
   If the varieties generated by classes $\clC_1$ and $\clC_2$ are locally finite, then the 
   variety generated by  the class  $\clC_1 \cup \clC_2$ is locally finite as well:  $\clC_1 \cup \clC_2$ is uniformly locally finite by Malcev criterion 
   given in Theorem \ref{Malcev73}.}
    %\(C\cong \rect{1}{\kappa}\) for some cardinal~\(\kappa.\)
    %The class of rectangle frames of the form \(\{\rect{1}{\kappa}\}_\kappa\) for all cardinals~\(\kappa\) is uniformly tunable.
  \end{proof}

\ISLater{Generalize for commutator}
\begin{lemma}\label{lem:rpp_transfer}
    If the product rpp criterion holds for~\(\LCom{\LS{4}[h]}{\LS5}\), then it also holds for~\(\LS{4.1}[h+1] \times  \LS5,\) for any \(h < \omega.\)
\end{lemma}
\begin{proof}
    Let~\(h < \omega\). Assuming the product rpp criterion for \(\LS{4}[h]\times  \LS5\), we will show it for~\(\LS{4.1}[h+1]\times \LS5.\)

    If~\(L\) is a locally tabular normal extension of~\(\LS{4.1}[h+1]\times \LS5,\) then it contains a product rpp formula by Proposition~\ref{prop:lt_implies_rpp}. 
    % For a formula \(\vf \in \ML(\Di)\), let \(t(\vf)\) be a translation of \(\vf\) that is compatible with Boolean connectives and satisfies \(t(\Di \varphi) = \Di_1 \varphi \lor \Di_2 \varphi.\)
    % Let \(L_0 = \{\varphi \mid [\varphi] \in L\}.\)
    % Then \(L_0\) is locally tabular.
    % By Theorem~\ref{thm:ShSh:RP}, the class of~\(L_0\)-frames satisfies \(\FORP_m\) for some~\(m < \omega.\)
    % By Proposition~\ref{prop:LT_implies_fmp+extensions}, \(L_0\) is Kripke complete, so \(\RP_m(\Di)\in L_0.\)
    % It follows that \(\RP_m(1,1) = [\RP_m(\Di)] \in L.\)
    
    Conversely, let~\(L\) be a normal extension of~\(\LS{4.1}[h+1]\times \LS5\) containing a product rpp formula.
    %If \(\LS5^2 \subseteq L,\) then \(L\) is locally tabular by Theorem~\ref{thm:S5Nick}.
    If   $L$ is an extension of  $\LS5^2$,  then it is a proper extension of $\LS5^2$, and so is locally tabular by Theorem~\ref{thm:S5Nick}.
    
    Otherwise the canonical frame $F_L$ of $L$ does not validate $\LS5^2$, since $\LS5^2$ is canonical \cite{Segerberg2DimML}. 
      By Proposition~\ref{prop:rpp-canon}, \(F_L\) validates a product rpp formula.
      By Proposition~\ref{prop:alpha_transfers_rpp},~\(\alpha(F_L)\) also validates a product rpp formula.
      Then by the assumption~\(\Log \alpha(F_L)\) is locally tabular.
      By Proposition~\ref{prop:alpha_transfers_lt}, so is~\(\Log F_L \subseteq L.\)
      Then~\(L\) is locally tabular.
\end{proof}\ISLater{DC REF to Seg}

The product rpp criterion holds for \(\LS{5} \times \LS5\). 
Indeed, $\LS{5}\times \LS{5}$ has no product rpp, and is not locally tabular. 
All extensions of $\LS{5}\times \LS{5}$ are locally tabular, and hence they have the product rpp by Proposition~\ref{prop:lt_implies_rpp}. 
It is well known that $\LS{5} \times \LS5=\LCom{\LS5}{\LS5}$  \cite{Segerberg2DimML}. 
Hence, we have 

\begin{theorem}\label{thm:rppCriterionS41}
    The product rpp criterion holds for \(\LS{4.1}[2]\times \LS5.\)
\end{theorem}

%I moved here

\subsection{Two dimensional tack}\label{subsec:Tack2}
\ISLater{

\IS{Sort it out: 
...

Let \(\TF_2(\kappa,\mu) = (\rect{\kappa}{\mu}) \oplus_1 \circ\).
\IS{Define}

Let \(\TF_2 = \TF_2(\omega,\omega)\) and \(\TL_2 = \Log(\TF_2).\)
\begin{claim}
    \(\TL_2\) is prelocally tabular.
\end{claim}
Idea: \(\LS4[h] \times \LS5 \subseteq \LS5^2\) or \(\TL_2\) \VS{what does it mean?}

\begin{proposition}[No proof known. Yet!]
Let $C$ be an  $\LS{4}\times \LS5$ cluster such that has no product rpp (not uniformly tunable?). 
Then  $C\toto \rect{m}{n}$ for all finite $m,n$.
\end{proposition}
\IS{Idea: 
If $(\LS{4}\times \LS5)[1]$  is complete w.r.t. $C$, we are done: 
in this case $C\toto F$ for any finite $\LS{5}\times \LS{5}$-frame. 
}

\begin{proposition}
  Let \(F \models [\LS4[h];\LS5]\) 
  \IS{Let \(F \models \LS{4}\times \LS5\) ?}
  be a rooted frame, \(h(F)> 1\), such that the root cluster \(C\) has no product rpp. Then:
  \begin{enumerate}
      \item \(F\toto C \oplus_1 \circ\);\IS{True}
      \item \(F\toto Tack_2(n,m)\) for all \(n,\,m < \omega.\)
      \IS{Would imply  Theorem \ref{thm:finalDream}  for Kripke complete logics.}
  \end{enumerate}
\end{proposition}

}  

}
%\subsection{A prelocally tabular extension of $\LS{4}\times \LS{5}$ of height 2}

%\subsection{Two-dimensional tack}

%\subsection{$\TL_2$ is prelocally tabular} 
\hide{
\IS{ Steps: 
\begin{enumerate}
    \item FMP
    \item Not locally tabular
    \item Semantics
    \item Kripke complete case 
    \item Canonical (Similarly to \cite{...}, we will identify a rectangle partition in the rooted cluster ....) 
\end{enumerate}
}

}

% The following theorem gives
A prelocally tabular logic $\LS{5}\times \LS{5}$ has height~\(1\).
In this subsection, we give another example of a
 prelocally tabular logic above $\LS{4}\times \LS{5}$, which has height~\(2\).

\hide{
We define the \emph{semi-ordered sum} \(F_1 \oplus_1 F_2\) of disjoint bimodal frames \(F_1\) and \(F_2\) as \(\sum_I F_i,\) where \(I = (\{1,2\},\le,Id\})\).
}
We define the \emph{semi-ordered sum} \(F \oplus G\) of disjoint bimodal frames \(F = (X,R_1,R_2)\) and \(G = (Y,S_1,S_2)\) as \((X\cup Y, R_1 \cup S_1 \cup (X \times Y), R_2 \cup S_2).\)
Let~\(\circ\) denote the reflexive, with respect to both relations, singleton \((\{0\},R,S)\) with \(R = S = \{(0,0)\}.\)

Let  $\TF(X,Y)=\Tsum{(\rect{X}{Y})}$. 
The \emph{two dimensional tack} is the  frame 
%$\Tsum{(\rect{\omega}{\omega})}$. 
$\TF(\omega,\omega)$.
%\(\TF(m,n)\) for \(m,\,n \le \omega\) is the bimodal frame \((\rect{m}{n}) \oplus  \circ.\)
Let \(\TL_2 = \Log{\TF(\omega,\omega)}.\)  
We show that  $\TL_2$ is canonical, has the finite model property, and is prelocally tabular.

\smallskip

Firstly, we identify some axioms of $\TL_2$, study its abstract 
frames, and then give a complete axiomatization.

Notice that the composition of relations $R=R_1\circ R_2$ in $\TF(\omega,\omega)$
is a preorder. 
Let $\Di$ denote the compound modality $\Di_1\Di_2$.

Observe that the   skeleton of $\TF(\omega,\omega)$ is a two-element chain. 
Hence, $\TL_2$ contains the axiom of height two for the compound modality $\Di$. 
Moreover, the height of the first relation in $\TF(\omega,\omega)$ 
is 2, so the fragment of $\TL_2$ in the first modality contains $\LS{4}[2]$.
Clearly, its fragment in the second modality is $\LS{5}$.
\hide{
contains $\LS{4}[2]\times \LS{5}$.  \VS{Why not \([\LS{4}[2],\LS5] \subseteq \TL_2 \subsetneq \LS{4}[2]\times \LS5\)? Maybe we can say that \(\TF(\omega,\omega) \cong\) a quotient of a product frame? } 
\hide{
\begin{proposition}
$\TL$ extends $(\LS{4}\times \LS{5})[2]$. 
\end{proposition}
}
}

Moreover, our observation about the skeleton being 
the two-element chain implies the following modally definable properties.
Clearly, the composition $R_1\circ R_2$ in $\TF(\omega,\omega)$
satisfies the Church-Rosser property, and hence, the formula $\Di \Box p\imp\Box \Di p$, 
which is an abbreviation for 
\begin{equation}\label{eq:TL-CR}
\Di_1\Di_2 \Box_1\Box_2 p\imp\Box_1\Box_2 \Di_1\Di_2 p,
\end{equation}
is a valid principle of $\TL_2$. 
In the combination with height 2, it follows that in every rooted 
$\TL_2$-frame, there is a top cluster.  
Next, observe that $R$ satisfies the condition 
\eqref{eq:s4.1-condition}. 
It follows that the formula $\Box\Di p\imp\Di\Box p$, 
which is an abbreviation for 
\begin{equation}\label{eq:TackMC}
\Box_1\Box_2\Di_1\Di_2 p\imp\Di_1\Di_2\Box_1\Box_2 p,      
\end{equation}
belongs 
to the logic $\TL_2$. 
In particular, it follows that in every rooted $\TL_2$-frame, 
the top cluster is a singleton. %\VS{We might need an explanation why the top cluster is unique}. 

Our next step aims to express the symmetry property in the bottom cluster of an arbitrary rooted $\TL_2$-frame (notice that in clusters of $\LS{4}\times \LS{5}$-frames, the symmetry of the first relation is not guaranteed)\IS{DC}.
Consider the following relativized version %$\ts{Sym}_2$ 
of the symmetry axiom: 
\ISLater{This may be known. We need to find the reference.}
\ISLater{
\IS{Perhaps, this is known. We need to check it.}
\VS{Follows from the finite height by Lemma~\ref{lem:opposite_arrows}.}}
\begin{equation}\label{eq:relativized-sym}
\Di_1\Box_1 q \con \neg q\con p\imp \Box_1 (\neg q\imp \Di_1 p) 
\end{equation}  
Readily, this formula is valid in $\TF(\omega,\omega)$.
Let $F$ be a point-generated $\TL_2$-frame of height 2. We claim that its rooted cluster $C$, considered as a frame-restriction, validates $\LS{5}\times \LS{5}$. \ISLater{Give axiomatization of $\LS{5}^2$}
 %If $C$ is a top cluster of $F$. Then (....).    
Indeed, assume that $x,y\in C$ are connected by the first relation $R_1$. 
Then consider a valuation $\theta$ in $F$ with $\theta(q)$ being  the top cluster, and $\theta(p)=\{x\}$. The validity of \eqref{eq:relativized-sym}  implies that $yR_1 x$, which proves the claim.  

\smallskip

Let $L$ be the extension of the commutator $[\LS{4}[2], \LS{5}]$ 
    with the axioms   \eqref{eq:TL-CR}, \eqref{eq:TackMC}, and \eqref{eq:relativized-sym}.
Putting our observations together, we obtain

   \begin{lemma}\label{lem:rootelTLframes}
   Let $F$ be a rooted   $L$-frame. Then $F$ is isomorphic to $C\oplus \circ$ for an $\LS{5}^2$-frame $C$, or to $\circ$. 
   \end{lemma}

Our next goal is to prove completeness. 

\ISLater{
For a $\Al$-frame $F=(X,(R_a)_\AlA)$, let $\Logu(F)$ denote the logic of $F$ in the language, 
extended with the universal modality, namely, the logic of 
the frame $(X,(R_a)_\AlA,X\times X)$. As usual,  for a class of frames, 
$\Logu(\clF)$ denotes the intersection of $\Logu(F)$ with $F\in\clF$.
\begin{lemma}
(interchangeable classes)    
\end{lemma}
\begin{proof}
 Follows from \cite[Lemma 4.6]{AiML2018-sums}. 
\end{proof}
}

\begin{lemma}\label{lem:LT-invariance}
Assume that $\Log(\clC)=\Log(\clD)=\LS{5}^2$ for classes
$\clC$, $\clD$  of frames. Then 
$\Log\{\Tsum{C}\mid C\in \clC \} =\Log\{\Tsum{D}\mid D\in \clD \}$. 
\end{lemma} 

This lemma follows from \cite{AiML2018-sums}.
The semi-ordered sum $\Tsum{F}$ is the sum of frames 
$F$ and %a reflexive (with respect to both relations) singleton  
the singleton  $\circ$
over the indexing frame $(2,\leq,\emp)$ in the sense of \cite[Definition 1]{AiML2018-sums}.  
%\VS{We already defined \(\circ\) above.}\IS{Resolved?}
Two classes are {\em interchangeable}, if they have the same logic in the language enriched with the universal modality.
It was shown  that in a class of sums, 
replacing interchangeable classes of summands does not change the logic of the sums.    
Classes $\clC$ and $\clD$ are interchangeable, since 
the universal modality is expressible as $\Di_1\Di_2$ in $\LS{5}^2$-frames. 
Hence $\Log\{\Tsum{C}\mid C\in \clC \} =\Log\{\Tsum{D}\mid D\in \clD \}$. 
See  Appendix, Section \ref{sub:sums}, for details. 

\hide{

%This is the case for classes, 
Classes $\clC$ and $\clD$ are interchangeable, since 
the universal modality is expressible as $\Di_1\Di_2$ in $\LS{5}^2$-frames.  
It follows from 
\cite[Lemma 4.6]{AiML2018-sums} that a formula is satisfiable in 
$\Tsum{C}$ for some $C\in \clC$  iff 
it is satisfiable in $\Tsum{D}$ for some $D\in \clD$. We provide the details in Appendix, Section \ref{sub:sums}  

}

\hide{
\begin{proof}
Any frame $\Tsum{F}$ is the sum of frames 
$F$ and %a reflexive (with respect to both relations) singleton  
the singleton  $\circ$
over the indexing frame $(2,\leq,\emp)$ in the sense of \cite[Definition 1]{AiML2018-sums}.  
%\VS{We already defined \(\circ\) above.}\IS{Resolved?}
We  apply a general observation about sums given in \cite[Section 4]{AiML2018-sums}. Two classes are {\em interchangeable}, if they have the same logic in the language enriched with the universal modality. %This is the case for classes, 
Classes $\clC$ and $\clD$ are interchangeable, since 
the universal modality is expressible as $\Di_1\Di_2$ in $\LS{5}^2$-frames.  
Now it follows from 
\cite[Lemma 4.6]{AiML2018-sums} that a formula is satisfiable in 
$\Tsum{C}$ for some $C\in \clC$  iff 
it is satisfiable in $\Tsum{D}$ for some $D\in \clD$. 
\ISLater{Improve}
\end{proof}
}

It follows from Lemma \ref{lem:LT-invariance}   that $\TL_2$ has an analog of the product fmp: 
\begin{corollary}\label{cor:TL-fmp}
$\TL_2$ is the logic of $\{\TF(m,m) \mid   m < \omega\}$.
\end{corollary}
\hide{
\begin{proof}   In the sense of Definition \ref{def:sum-over-frame}, 
the frame $\TF(\kappa,\mu)=(\rect{\kappa}{\mu})\,\oplus\,\circ$ is the sum of frames 
$\rect{\kappa}{\mu}$ and a reflexive (with respect to both relations) singleton  
over the indexing frame $(2,\leq,\emp)$.  
We have 
$\Log(\rect{\omega}{\omega})=\Log\{\rect{m}{m}\mid  m < \omega\}$. This equality also holds if the modal language is enriched with the universal modality, since it is expressible as $\Di_1\Di_2$ in our case.  
Now it follows from 
\cite[Lemma 4.6]{AiML2018-sums}, that a formula is satisfiable in 
 $(\rect{\omega}{\omega}){\,\oplus\,}\circ$ iff it is satisfiable 
 in $(\rect{m}{m}){\,\oplus\,}\circ$ for some $m$.  
\ISLater{Improve}
\end{proof}
\IS{Re-define this staff; move sums somewhere; make  a lemma for sum of two frames?}

}

\begin{theorem}\label{thm:TL2-canon}~
\begin{enumerate}
    \item $\TL_2$ is the extension of the commutator $[\LS{4}[2],\LS{5}]$
    with the axioms \eqref{eq:TL-CR},  \eqref{eq:TackMC}, and \eqref{eq:relativized-sym}.
    \item $\TL_2$ is canonical.
\end{enumerate}
\end{theorem}
\begin{proof} 
Let $L$ be the mentioned axiomatic extension of $[\LS{4[2]}, \LS{5}]$. 

Firstly, we show that $L$ is canonical. 
The axioms of $[\LS{4}[2],  \LS{5}]$ are known to be canonical, in particular, canonicity of the finite height formula is given in \cite{Seg_Essay}.\ISLater{theorem}
 It is immediate that  \eqref{eq:relativized-sym} is equivalent 
to the Sahlqvist formula 
$$
\Di_1\Box_1 q \con p\imp  \Box_1( q \vee \Di_1 p)\vee p, 
$$
and so is canonical. 
The formula $\Di \Box p \imp  \Box \Di p$ is also Sahlqvist.

The formula $\Box\Di p \imp \Di \Box p$ 
is not a Sahlqvist formula. However, it is well-known that it is canonical in the transitive case \cite[Theorem 6.4]{SegerbS41}, which, in our setting, pertains to 
the composition $R_1\circ R_2$.  Hence, we have a singleton cluster above every point in the canonical frame of $L$. 

It follows that $L$ is canonical.

\smallskip 

Let us show that $L=\TL_2$. 
Since $L$ is canonical, it is Kripke complete, and so, in view of  Lemma \ref{lem:rootelTLframes}, 
$L$ is the logic of the class $\{\Tsum{C}\mid C\mo \LS{5}^2\}$.
From Lemma \ref{lem:LT-invariance}, 
the logic of this class is $\TL_2$.
\end{proof}

\begin{lemma}
$\TL_2$ is not locally tabular.     
\end{lemma}
\begin{proof}
The frame $\TF(\omega,\omega)$ contains $\rect{\omega}{\omega}$ as a subframe. 
Now the statement follows from Proposition \ref{prop:LF-for-subframess}.
\end{proof}

%For a logic $L$, let $F_L$ be its canonical Kripke frame \cite[Section 4.2]{BDV}.  
\ISLater{Improve: make two lemmas: Glivenko, Nick}

\begin{lemma}\label{lem:ext-TL-complete}
Let $\clF$ be a class of bimodal frames such that $\Log\clF$ is a proper 
extension of $\TL_2$. Then $\Log\clF$
is locally tabular.
\end{lemma}
\begin{proof}
Let \(\clG\) be the class of point-generated subframes of frames in $\clF$.
By Lemma \ref{lem:rootelTLframes}, every frame in $\clG$ is isomorphic to  $C\oplus \circ$ for an $\LS{5}^2$-cluster $C$. 
Let $\clC$ be the class of all such frames. 
Since $\Log\clF=\Log\clG$ is a proper 
extension of $\TL_2$,  then, in view of Lemma \ref{lem:LT-invariance}, $\Log(\clC)$  is a proper extension of 
$\LS{5}^2$. By Theorem \ref{thm:S5Nick}, 
the logic $L$ of this class is locally finite. So $\clC$ is a uniformly tunable class. Clearly, the top clusters in $\TL$-frames also form a uniformly tunable class, since they are singletons. Now the statement follows from Theorem 
\ref{thm:supple-clusters-crit}.
\end{proof} 

\begin{remark}
For a Kripke complete proper extension $L$ of $\LS{5}^2$, local tabularity can be obtained directly,  as a  corollary of  
Theorem \ref{thm:criterion-frames-general}.
%and observations in  Section \ref{subs:Prod}. 
See Appendix, Section \ref{subs:Prod}. 
\end{remark}

\begin{theorem}
$\TL_2$ is prelocally tabular. 
\end{theorem}
\begin{proof}
Let $L$ be a non-locally tabular extension of $\TL_2$. 
We show that $L=\TL_2$.  

For some finite $k$, the $k$-generated canonical Kripke frame $F_L$ of $L$
is infinite.\ISLater{define}

Let us list some properties of $F_L$.  Since $\TL_2$ is canonical (Theorem \ref{thm:TL2-canon}), and $L$ contains $\TL_2$, $F_L$ is a $\TL_2$-frame.\ISLater{Earlier argument was: We have $\Log(F_L)\subseteq L$, so $\Log(F_L)$ is not locally tabular. I am not sure if $\Log(F_L)\subseteq L$.}  
The top clusters in $F_L$ are singletons. Let $\clC$ be the family of non-top clusters in $F_L$. 
In view of Lemma \ref{lem:rootelTLframes}, $\clC$ are $\LS{5}^2$-frames.  
We also have:
\begin{equation}\label{eq:tack-clusterD}
\text{For  $C\in \clC$, there is a unique cluster $D$ above $C$, and $D$ is a singleton.     }
\end{equation} 

\hide{
The powerset algebra of $F_L$ contains an infinite $k$-generated subalgebra, and so 
$\Log(F_L)$ is not locally tabular.
Since $F_L$ is a frame of  finite height, by Theorem \ref{thm:supple-clusters-crit},
the logic $L'$ of  $\clC$ is not locally tabular. Hence,  $L'$ is $\LS{5}^2$. 
}
Now we claim that $\clC$ contains an infinite cluster. 
For the sake of contradiction, suppose that all clusters in $\clC$ are finite. 
In this case, all point-generated subframes if $F_L$ are finite, and it is straightforward 
that $\Log(F_L)=L$. But in this case, $L$ is locally tabular by Lemma \ref{lem:ext-TL-complete}. This contradiction proves that $\clC$ contains an infinite cluster.

\ISLater{Perhaps, We can use the RPP-criterion} %\VS{Theorem~\ref{thm:supple-clusters-crit}?}

\ISLater{In the submitted version: ``$F_L$ contains an infinite cluster $C$.'' But in fact, I do not see why.  (Is it true at all?) It contains clusters of arbitrary size, this is true. A possible fix: build p-morphism from each of them. However, in this case  
we cannot quote Nick's lemma directly, we need to restate it. Again:) 
Perhaps, a simpler fix is to use an rpp-criterion. But how do we know that $\TL_2$ conatins 
the product? }

Let $C$ be an infinite cluster in $\clC$, and let $D$ be the cluster above $C$. Put $X=C\cup D$. Clearly, \ISLater{Details} 
the restriction of $F=F_L\restr X$ is a generated subframe of $F_L$.

Consider the extension $\TL_2[1]$   of $\TL_2$ with the axiom of height 1, which makes 
the composite relation $R_1\circ R_2$ equivalence. Clearly, $\TL_2[1]$ is locally tabular (even tabular, since this is the logic of the singleton $\circ$). 
It follows that the set $T$ of points in top clusters in $F_L$
is definable \cite[Section 4]{Glivenko2021}: for a formula $\psi_T$, 
\begin{equation}
\psi_T\in x \text{ iff } x\in T    
\end{equation}

Hence, for every $x$ in $X$:
\begin{equation}\label{eq:tack:defTop}
  x\in C \text{ iff } \psi_T\notin x
\end{equation}

%Let $\theta$ be the canonical valuation, that is, 
Let $\Psi$ be the set of all bimodal formulas in $k$ variables.  
For a formula $\psi\in \Psi$, let $\theta(\psi)=\{x  \in X \mid \psi\in x\}$, 
$\eta(\psi)=\{x  \in C \mid \psi\in x\}$.

Let $A$ denote the powerset modal algebra of the cluster-frame $F\restr C$, and let $B$ 
be its subalgebra generated by the sets $\eta(p_i)$, $i<k$.

By a straightforward induction on the structure of formulas, 
we have: 
\begin{equation}\label{eq:tack:defSets}
\text{ For every $b\in B$, there is $\psi\in\Psi$ such that $b=\theta(\psi)\cap C$.}
\end{equation}

Since $C$ is infinite, $B$ is infinite as well: %indeed, in canonical frames, 
any two distinct points in $C$ are separable by some formula.\ISLater{Details} 
Since $C$ is an $\LS{5}^2$-frame,   $B$ is an $\LS{5}^2$-algebra. \ISLater{Define in prel}
Let $A_m$ denote the powerset modal algebra of the rectangle $\rect{m}{m}$.
It follows from \cite[Claim 4.7]{NickS5} that all finite $A_m$ are embeddable in $B$. Let $A_m\cong B_m\subset A$, and let $f_m$ be the corresponding p-morphism  $F\restr C\toto \rect{m}{m}$.

\ISLater{\VS{
We proceed by showing that the formulas from~\(\Psi\) distinguish the elements of~\(B_m\).
Indeed, let~\(b \in B_m\).
By~\eqref{eq:tack:defSets} there exists a formula~\(\psi\in \Psi\) such that \(b = \theta(\psi)\cap C.\)
By~\eqref{eq:tack:defTop}, the formula \(\psi\wedge \neg \psi_T\) characterizes~\(b\) in~\(F.\)
}}
Let $b_{ij}$, $i,j<m$ be the atoms of $B_m$. By \eqref{eq:tack:defSets}, 
there are formulas $\psi_{ij}$ such that $b_{ij}=\theta(\psi_{ij})\cap C$. 
By \eqref{eq:tack:defTop}, $b_{ij}$ is definable in $F$ by the formula 
$\psi_{ij}\wedge \neg \psi_T$. 
Let $g_m:F\to \TF{(m,m)}$ 
%$g_m:F\to \Tsum{\rect{m}{m}}$ 
be the extension of $f_m$ that maps the top to the top. 
Clearly, $g_m$  is a p-morphism. Since all preimages $g_m^{-1}(i,j)$
%\VS{\(g_m\inv\)?}
are definable in $F$ by $\psi_{ij}\wedge \neg \psi_T$,\ISLater{formally, indexing is not defined} 
the algebra of $\TF(m,m)$ 
%$\Tsum{\rect{m}{m}}$ 
is embeddable in the algebra of sets $\theta(\psi)$, \(\psi\in\Psi\). The latter algebra is an $L$-algebra, since $F$ is a generated subframe of $F_L$. 

So the algebras of all frames $\TF(m,m)$  are $L$-algebras, and due to Corollary \ref{cor:TL-fmp}, 
$L\subseteq \TL_2$. Consequently, $L=\TL_2$. 
\end{proof}

%\newpage 
%\input{tack2_alternative_and_remark.tex}

\hide{

 \bigskip 
 ....

Let $\clV=\{\eta(\psi) \mid \psi\in\Psi\}$. Since $C$ is infinite,  $\clV$ is infinite as well.\IS{details?} 

We claim that $B$ is finitely generated.

({\em a component}, in terms of \cite{NickS5}.)

and let $\clV$ be all these evaluations of formulas in $k$ variables. Then $\clV$ is infinite.\ISLater{details}

Let $F_C$ be the restriction of $F$ to $C$, and let 
$A(F_C)$ be its powerset algebra. It is straightforward that 
$A(F_C)$ contains a finitely generated infinite subalgebra $B$ of .....\IS{Details}
It follows from \cite[Claim 4.7]{NickS5} that 
in this case, there is a subalgebra of definable .... 
which is isomoriphic to the powerset algebra of $\rect{m}{m}$...
}

 \hide{

\begin{lemma}
Let $L$ be an extension of $\TL_2$. If 
 \(\Log{F_L} = \TL_2\), then $L=\TL_2$. 
\end{lemma}
\begin{proof}  
$F_L=(X,R_1,R_2)$. 
Since $\Log{F_L}=\TL_2$, no product rpp formula is valid in $F_L$.\IS{Explain better why. Use what?} 
These formulas are canonical, and so none of them belong to $L$. 
It follows that for every $n$, there are 
formulas $\psi_1,\ldots, \psi_n$ and definable subsets 
$V_i=\{x \text{ in }X \mid \psi_n \in x   \}$  of $F_L$
such that for distinct $V_i,V_j$, for all $x\in V_i$, $y\in V_j$, 
$(x,y)\notin R_1$ and $(x,y)\notin R_2$. 
\ISLater{Details?} 
It follows from \cite[Claim 4.7]{NickS5} that 
in this case, there is a subalgebra of   
\IS{....}
\end{proof}
 
 }
%\newpage
%\pagebreak
\section{Conclusion}

\subsection{Summary of the results}

\begin{enumerate}
    \item The main result of this paper is the  criteria obtained in Section \ref{sec:criteria}. 
    \item In Section \ref{sec:examples}, the criteria are applied  to describe new families of locally tabular products. 
    In particular, we %generalized the Segerberg-Maksimove criterion for some classes of products й(?)  and
generalized  results from \cite{Shehtman2012,Shehtman2018}.
%for the products with logics containing \Box^...  
\item In Section \ref{sec:pfmp}, we discussed the product finite model property. In particular, we showed 
  that the local tabularity of a product logic does not imply the product fmp, even in the case of height 3. 
\item We showed in Theorem \ref{thm:rppCriterionS41} that the product rpp formula gives an  axiomatic criterion of local tabularity  for all extensions of 
$\LS{4}.1 [ 2 ]\times \LS{5}$. 
\item Finally, we described a prelocally tabular extension of $\LS{4}\times\LS{5}$ of height 2.  
\end{enumerate}

\subsection{Open problems}

\subsubsection{Criterion for extensions of products}
According to Theorem \ref{thm:criterion-frames-general}, 
local tabularity of the factors and a product rpp formula are equivalent to the local tabularity of the product. In particular, in the case of  transitive logics of finite height, 
local tabularity of the product is equivalent to the product rpp due to Corollary \ref{cor:criterion-SegMaks}. 
These facts do not give a criterion 
for normal extensions of a product.
Unlike these facts, Theorem 
\ref{thm:rppCriterionS41}
gives an axiomatic criterion for all extensions of $\LS{4}.1 [ 2 ]\times \LS{5}$.
We conjecture that Theorem \ref{thm:rppCriterionS41} can be generalized for a wider class of logics.\footnote{
In a very recent preprint  \cite{Meadors_MS4_Arxiv}, 
it was announced that 
Theorem \ref{thm:rppCriterionS41} 
%the product rpp criterion 
extends to logics above  \(\LS4[2]\times \LS5\). 
\hide{
    By Theorem~\ref{thm:rpp_transfer}, this result implies that the criterion also holds for~\(\LS{4.1}[3]\times \LS5.\) 
    
    After the first version of our manuscript had appeared online, 
    the result that the product rpp criterion  holds for~\(\LS4[2]\times \LS5\) was announced 
    in a recent preprint  \cite{Meadors_MS4_Arxiv}.
    %By Theorem~\ref{thm:rpp_transfer}, this result implies that the criterion also holds for~\(\LS{4.1}[3]\times \LS5.\)

    }
}    

\smallskip 
{\bf Problem.} 
Let $S4[h]$ be the extension of $\LS{4}$ with the axiom of finite height $B_h$.
Does the  product rpp criterion hold for every extension $L$ of $\LS{4}[h]\times \LS{5}$?  

%, $L$ contains a formula of finite height. 
 If not, what is the largest such $h$?     

 \smallskip

\hide{
following equivalence hold for every extension $L$ of $\LS{4}[h]\times \LS{5}$? 
\begin{center}
$L$ contains a product rpp formula iff  $L$ is locally tabular. 
\end{center} 

}

\hide{
\IS{Move to remark}
\begin{remark}
    After the first version of our manuscript had appeared online, 
    the result that the product rpp criterion  holds for~\(\LS4[2]\times \LS5\) was announced 
    in a recent preprint  \cite{Meadors_MS4_Arxiv}.
    By Theorem~\ref{thm:rpp_transfer}, this result implies that the criterion also holds for~\(\LS{4.1}[3]\times \LS5.\)
\end{remark}

\IS{Or:
 A generalization of Corollary \ref{cor:rppCriterionS41}...
}
}

\improve{
\ISH{ 
Since the converse implication is given by Corollary \ref{cor:criterion-logics-general}, 
the affirmative answer would give an axiomatic criterion of local tabularity above $\LS{4}\times \LS{5}$. -- does Corollary \ref{cor:criterion-logics-general} really applies? 
}
}
\extended{ 
\ISH{Seems that the same trick works for the logic of the flipped tack}
} 
    
\subsubsection{Prelocal tabularity in products}

It is a well-known open problem whether every non-locally tabular 
unimodal logic is contained in a prelocally tabular \cite[Problem 12.1]{CZ}. 
For the case of logics  above $\LS{4}\times \LS{4}$, 
this question is also open. We conjecture that this is true for logics above $\LS{4}\times \LS{5}$. 

\smallskip

\subsubsection{Local tabularity in the symmetric case}

One of the central tools for our results was the criterion given in Theorem \ref{thm:supple-clusters-crit} \cite{LocalTab16AiML}, which, informally, says, that 
under the necessary conditions of pretransitivity and finite height, only local tabularity on clusters matters. While this criterion has  many applications,  there is an obvious limitation for this approach. 
Namely, even in the unimodal case, if the relation is symmetric, then every point-generated frame consists of one cluster, and no additional structure is given by the skeleton construction. 
Hence, it is of definite interest to study local finiteness of symmetric relations. 
Under the necessary conditions of local finiteness, this means that we are interested in graphs of finite diameters. No axiomatic criterion 
is known for this case,  and it is  not guaranteed that 
an explicit axiomatic characterization of local finiteness is possible in this case.
\ISLater{
Another characterization of this case is given in \cite{KrachtKovalsky}: ..... (semisimple staff). 
In a recent paper \cite{GuramMS4}, it was shown that symmetric logics can be interpreted in 
products of height 3 (DC details). So the question about locally finite symmetric logics is not simpler than the one on products above (details). And of course, there is no garrantee that 
an explicit characterization of local finiteness is possible by providing a recursive set of modal formulas. 
}

\ISLater{
(Some old staff:

By a straightforward induction on $n$, Corollary  \ref{cor:criterion-logics-general} generalizes to the following fact. 
\begin{corollary} Let $L_1,\ldots, L_n$  be Kripke complete consistent logics.
\ISH{Disjoint languages?}
TFAE:
\begin{enumerate}[(i)]
\item $L_1\times (L_2 \times (L_3 \times \ldots \times L_n)\ldots)$  is locally tabular;
\item $L_1,\ldots,L_n$ are locally tabular and at most one of them has
no bounded cluster property.
\end{enumerate}    
\end{corollary}

\ISH{Associativity and stuff}

\ISH{Product fmp and tensor product}  

\ISH{Generalization of \ref{thm:some-axiomatic-crit}}

}

\section{Acknowledgements}
We are grateful to Guram 
Bezhanishvili 
and Valentin Shehtman for valuable discussions. 
%The authors would like to thank the anonymous reviewers for their  comments on an earlier version of the manuscript. 
%\IS{We do not need to do it here}

% \bibliographystyle{asl}
\bibliographystyle{amsalpha}

\bibliography{refs}
\newpage

\section{Appendix}
\subsection{Proof of Theorem \ref{thm:supple-clusters-crit}}

\begin{lemma}\label{lem:generated-subframe-tunable}
  Let $\AlA$ be finite, \(F\)  an \(\AlA\)-frame such that:
  \begin{enumerate}[(a)]
    \item \(\dom F = X_1\cup X_2\) for disjoint $X_1$ and $X_2$, and
    \item \(F\restr X_1\) is a generated subframe of \(F,\) and   
    \item \(F\restr X_i\) is \(f_i\)-tunable for some monotone \(f_i:\:\omega\to \omega,\) \(i = 1, 2.\)
  \end{enumerate} 
  Then \(F\) is \(g\)-tunable for 
  \[
    g(n) = f_1(n) + f_2\left(n\cdot2^{f_1(n)\cdot |\AlA|} \right).
  \]
\end{lemma}
\begin{proof}
Let \(\clV\) be a partition of \(\dom F,\,|\clV| = n < \omega.\) Then \(\clV\) induces partitions \(\clV_1,\,\clV_2\) of \(X_1\) and \(X_2,\) respectively. There exists a refinement \(\clU_1\) of   \(\clV_1\) such that $\clU_1$ is tuned in $F_1$  and 
\begin{equation}\label{eq:clusters-1}
|\clU_1| \le f_1(|\clV_1|) \le f_1(n). 
\end{equation}

For $\Di\in\AlA$, define \(\alpha_\Di:\:X_2  \to 2^{\clU_1}\) by letting 
  \[
    \alpha_\Di(x) = \{U\in \clU_1 \mid x\in R_\Di\inv[U]\}.
  \]
Let $\sim_\Di$ be the equivalence induced by $\alpha_\Di$; it has at most 
$2^{|\clU_1|}$ classes.
Hence, the equivalence $\sim\; = \;\bigcap\{ \sim_\Di\mid \Di\in \Al\}$ has at most 
$2^{f_1(n)\cdot |\AlA|}$ classes.    
Let $\equiv$ be the equivalence on $X_2$ whose quotient is 
$\clV_2$. 
Put $\clS=\fact{X_2}{(\sim\cap \equiv)}$. 
We have 
\begin{equation}\label{eq:clusters-2}
|\clS| \le |\clV_2| \cdot  2^{|\clU_1|\cdot   |\Al|}  \le n \cdot 2^{f_1(n)\cdot |\Al|},  
\end{equation}
where the latter inequality holds by \eqref{eq:clusters-1}.
Hence, there exists a refinement \(\clT\) of \(\clS\) of size at most 
   \(f_2\left(n\cdot 2^{f_1(n)\cdot |\AlA|}\right)\), which is tuned in $F_2$. 

Consider the partition $\clU=\clU_1 \cup \clT$ of $F$. Clearly, it refines 
$\clV$, and it has at most $g(n)$ elements in view of 
\eqref{eq:clusters-1} and \eqref{eq:clusters-2}.
So it remains to check that $\clU$ is tuned in $F$.

   Let \(xR_{\Di} y\) for some \(x,\,y\in X\) and \(\Di\in \AlA\), and let \(x'\in [x]_\clU.\) 
   We will show that \(x' R_\Di y'\) for some \(y'\in [y]_\clU\).  

   Assume that \(x\in X_1\). Then \(y\in X_1\), since \(F\restr X_1\) is a generated subframe of $F$. In this case, \(x\) and \(x'\) belong to the same element of \(\clU_1.\) Since \(\clU_1\) is tuned in $F_1$, \(x' R_{\Di} y'\) for some \(y'\in [y]_{\clU_1}=[y]_\clU\), as desired.

   Now assume \(x\in X_2\). Then \(x'\in [x]_\clT\), so \(\alpha_\Di(x) = \alpha_\Di(x').\) If 
   \(y\in X_1\), then \([y]_{\clU_1}\in \alpha_\Di(x) = \alpha_\Di(x')\), so  $x' R_\Di y'$
   for some $y'\in [y]_{\clU_1}=[y]_\clT$,
   as desired. If \(y\in X_2\), then \(x' R_{\Di} y'\) for some \(y'\in [y]_\clT=[y]_\clU\), since \(\clT\) is tuned in $F_2$.  
\end{proof}

%Theorem~\ref{thm:supple-clusters-crit} follows from Proposition~\ref{prop:supple-clusters-crit-only-if}, Proposition~\ref{prop:supple-clusters-crit-if} and Theorem~\ref{thm:LFviaTuned}.

%\begin{proof}[Proof of theorem \ref{thm:supple-clusters-crit}]
\noindent
{\bf Proof of Theorem \ref{thm:supple-clusters-crit}.}
`Only if'. Finite height follows from Theorem \ref{thm:1-finite-to-m-h}.  
The logic of subframes of frames in $\clF$ is contained in the logic of the clusters $\clusters{\clF}$
and is locally tabular by Proposition \ref{prop:LF-for-subframess}.  Hence, the logic of $\clusters{\clF}$ is locally tabular. 

`If'.
Let $\clF_h=\{F\in \clF\mid h(F)\leq h\}$, $1\leq h<\omega$.  
By induction on $h$ we show that  $\clF_h$ is $g_h$-tunable for some monotone $g_h:\omega\to\omega$. 
Let $h=1$. The class $\clF_1$ consists of disjoint sums of frames in $\clusters{\clF}$, so $\Log\clF_1=\Log\clusters{\clF}$. Since $\Log\clF_1$ is locally tabular,  the class $\clF_1$ is $f$-tunable for some $f$ by Theorem~\ref{thm:LFviaTuned}. Put $g_1(n)=\max\{f(i)\mid i\leq n\}$; clearly, $\clF_1$ is $g_1$-tunable. 
Let $h>1$. 
%The class $\clusters{\clF}$ is $f$-tunable for some $f$ by Theorem~\ref{thm:LFviaTuned}. 
By induction hypothesis, 
$\clF_{h-1}$ is $g_{h-1}$-tunable for some monotone $g_{h-1}$. 
Put $g_h(n) = g_1(n) + g_{h-1}\left(n\cdot2^{g_1(n)\cdot |\AlA|} \right).$ 
Consider $F\in \clF_{h}$. 
Let $F_1$ be the restriction of $F$ on its maximal clusters, 
$X_1=\dom F_1$, $X_2=(\dom F){\setminus}X_1$. If $X_2 =\emp$,\improve{\ISH{Here we pay a price for not considering empty frames}} then $F\in\clF_1$, and so is $g_1$-tunable; hence, $F$ is $g_h$-tunable. If $X_2 \neq \emp$, $F$ is $g_h$-tunable by Lemma   \ref{lem:generated-subframe-tunable}.

Since $\clF=\clF_h$ for some $h$,  $\clF$ is uniformly tunable, and hence its logic is locally tabular by  Theorem~\ref{thm:LFviaTuned}.
\qed
%\end{proof}

 \smallskip

\subsection{Sums}\label{sub:sums}  

\begin{definition}\label{def:sum-over-frame}
Consider an $\Al$-frame $I=(Y,(S_\Di)_{\Di\in \Al})$
and a family $(F_i)_{i\in Y}$ of $\Al$-frames
$F_i=(X_i,(R_{i,\Di})_{\Di\in \Al})$.
The {\em sum
 $\LSum{I}{F_i}$ of the family $(F_i)_{i\in Y}$ over~$I$}
 is the $\Al$-frame $(\bigsqcup_{i \in Y} X_i, (R^\Sigma_\Di)_{\Di\in \Al})$,
 where
 $\bigsqcup_{i\in Y}{X_i}=\bigcup_{i\in Y}(\{i\}\times X_i)$,\extended{is the  disjoint union of sets $X_i$} and
$$(i,a)R^\Sigma_\Di (j,b) \quad \tiff \quad (i = j \text{ and }a R_{i,\Di} b) \text{ or } (i\neq j \text{ and } iS_\Di j).$$
For  classes $\clI$, $\clF$  of $\Al$-frames,
let $\LSum{\clI}{\clF}$  be the class of all sums
$\LSum{I}{F_i}$ such that $\frI \in \clI$ and  $F_i\in \clF$ for every $i$ in $\frI$.
\end{definition}

\begin{theorem}\cite{LTViaSums2022}\label{thm:sum}
Let $\clF$ and $\clI$ be classes of $\Al$-frames.
If the   logics $\Log{\clF}$ and
$\Log{\clI}$ are locally tabular, then  $\Log{\LSum{\clI}{\clF}}$  is locally tabular as well.
\end{theorem}

\begin{definition}
Two classes of frames  are said to be 
{\em interchangeable}, if they have the same logic in the language enriched with the universal modality. 
\end{definition}

The following fact follows from \cite[Lemma 4.6 and Lemma 4.8]{AiML2018-sums}.
\begin{lemma}\label{lem:interch} 
Let $I$ be an $\Al$-frame, and for each $i\in I$, let $\clF_i$ and $\clG_i$ be two interchangeable families of $\Al$-frames.   \ISLater{What is our meta-theory? Is it fine with ZF?}
Let $\clK_1$ be the class of sums $\LSum{I}{F_i}$ with $F_i\in \clF_i$ for each $i\in I$, and let 
 $\clK_2$ be the class of sums $\LSum{I}{G_i}$  with  $G_i\in \clG_i$ for each $i\in I$. 
Then $\Log \clK_1 =\Log \clK_2$. 
\end{lemma}
\hide{
\begin{proof}
This lemma follows from  
\cite{AiML2018-sums}.    \IS{To do} 
\end{proof}
}

\begin{proof}[Proof of Lemma \ref{lem:LT-invariance}]
Any frame $\Tsum{F}$ is the sum of  
$F$ and %a reflexive (with respect to both relations) singleton  
the singleton  $\circ$
over the indexing frame $I=(2,\leq,\emp)$ in the sense of 
Definition \ref{def:sum-over-frame}.
%\cite[Definition 1]{AiML2018-sums}.    
Classes $\clC$ and $\clD$ are interchangeable, since 
the universal modality is expressible as $\Di_1\Di_2$ in $\LS{5}^2$-frames.  
Now the statement  follows from Lemma \ref{lem:interch}.
\end{proof}

\smallskip

\subsection{Productivization}\label{subs:Prod}  
 \begin{definition}\label{def:tilde}
  Let \(F = (X,\,(R_\Di)_{\Di\in A})\) be an \(\AlA\)-frame, where all $R_\Di$ are preorders. Define the equivalence relation
  $
    {\sim}=\bigcap_{\Di\in A} (R_\Di\cap R_\Di\inv).
  $
  The quotient frame $\fact{F}{\sim}$ is denoted  \(\widetilde{F}\).   
  Given a class \(\clF\) of \(\AlA\)-frames, let %\(\quotfr{\clF}\) denote the class
  \(\quotfr{\clF} = \{\quotfr{F} \mid F\in \clF\}.\)
 \end{definition}
\ISLater{Where to move? -  Note that if \(F=(X,R_1,R_2)\) is an \(\LS{5}^2\)-frame, then \({\sim}\) is $R_1\cap R_2$.}
   \hide{
  \(= (\fact{X}{\sim},\,(\quotfr{R_\Di})_{\Di\in\Al})\) the minimal filtration\ISH{do we need this term?} of \(F\) with respect to \(\sim:\)
  \[
    [a] \quotfr{R_{\Di}} [b] \iff \exists a'\in[a]\ \exists b'\in [b]\ \left(a' R_{\Di} b' \right),\quad \Di\in \AlA.
  \]
  }
  
\begin{lemma}\label{lem:p-morph-sim}
  If the relations of \(F\) are
  %transitive, 
  preorders,
  then \(F\toto \quotfr{F}.\)
\end{lemma}
\begin{proof}
For any frame, the quotient map is a surjective homomorphism. To check the back condition, assume that \([a]\quotfr{R_\Di}[b]\) for some \(a,\,b\in X\) and \(\Di\in \AlA.\) Then \(a' R_\Di b'\) for some \(a'\sim a\) and \(b'\sim b.\) By the definition, \(a  R_\Di a'\) and \(b' R_\Di b,\)  so \(a  R_\Di a' R_\Di b' R_\Di b ,\) hence \(a R_\Di b\) by transitivity.
\end{proof}

\begin{lemma}\label{lem:quotient-uniformly-tunable}
  Let \(\clF\) be a class of frames such that in every $F\in\clF$, all relations are 
  %transitive. 
  preorders.
  Then \(\Log\quotfr{\clF}\) is locally tabular iff \(\Log\clF\) is locally tabular.
\end{lemma}
\begin{proof}
%Short proof via sums
The `if' direction is trivial, since $\Log\clF\subseteq \Log\quotfr{\clF}$ by Lemma \ref{lem:p-morph-sim}.

`Only if'.   Consider \(F=(X,\,(R_\Di)_{\Di\in \AlA})\in \clF.\) Let \(\sim\) be defined as in Definition~\ref{def:tilde}, and denote \(Y = X/{\sim}.\) It is straightforward that  $F$ is isomorphic to the  sum $\LSum{\quotfr{F}}({F\restr U})$, where \(U\subseteq X\) ranges over the 
equivalence classes in \(Y\). It follows that $\Log\clF$ is included 
in $\Log{\LSum{\quotfr{\clF}}{\clG}}$, where \(\clG\) is the class of all frames~\(F\restr U\) where \(F\in \clF\) and \(U \in \dom F/{\sim}.\) In frames in $\clG$, every relation is universal, hence every partition is tuned, and so its logic is locally tabular by 
Theorem~\ref{thm:LFviaTuned}. Since \(\Log \quotfr{\clF}\) is also locally tabular, $\Log\clF$ is locally tabular by Theorem~\ref{thm:sum}.
\end{proof}
\hide{
\begin{proof}[Old Proof via sums]
  The `if' direction is trivial, since $\Log\clF\subseteq \Log\quotfr{\clF}$ by Lemma \ref{lem:p-morph-sim}.

  `Only if'.
  Consider \(F=(X,\,(R_\Di)_{\Di\in \AlA})\in \clF.\) Let \(\sim\) be defined as in Definition~\ref{def:tilde}, and denote \(Y = X/{\sim}.\)
  
  We claim that \(F \cong \LSum{\quotfr{F}}({F\restr U})\), where \(U\subseteq X\) ranges over the equivalence classes in \(Y\). Define \(f:\:X \to \bigsqcup_{U \in Y} U\) by \(f(a) = ([a]_{\sim},\,a).\) It is trivial that \(f\) is injective. If \((U,\,a)\in \bigsqcup_{U \in Y} U\) then \(a\in U\subseteq X\) by the definition, so \(U = [a]_{\sim},\) hence \((U,\,a) = f(a).\) Then \(f\) is surjective.

  Suppose \(a,b\in X,\,a R_\Di b.\) If \(a \sim b\) then \(f(b) = ([a]_{\sim},\,b),\) so \(f(a)R^\Sigma_\Di f(b).\) Otherwise~\([a]_{\sim} {\ne} [b]_{\sim};\) observe that \([a]_{\sim} \quotfr{R_\Di} [b]_{\sim}\) by the definition of \(\quotfr{R_\Di},\) so~\(f(a) R^\Sigma_\Di f(b).\) 

  Conversely, let \(f(a) R^\Sigma_\Di f(b).\) If \([a]_{\sim} {=} [b]_{\sim}\) then \(a R_\Di b;\) if \([a]_{\sim}{\ne} [b]_\sim\) then \([a]_\sim \quotfr{R_\Di} [b]_\sim.\) By the definition of~\(\quotfr{R_\Di},\) there exist \(a'\sim a\) and \(b'\sim b\) such that \(a' R_\Di b'.\) By the definition of \(\sim,\) \(a R_\Di a'\) and \(b' R_\Di b,\) so by transitivity \(a R_\Di b.\) 

  The isomorphism implies that \(\Log\clF = \Log{\LSum{\quotfr{\clF}}{\clG}}\)\IS{I am not sure about one of the inclusions} where \(\clG\) is the class of all frames~\(F\restr U\) where \(F\in \clF\) and \(U \in \dom F/{\sim}.\)
  In any frame of \(\clG,\) the relations are transitive and symmetric by the definition of \(\sim\). Then \(\clG\) is uniformly tunable since on the frames of \(\clG\) any partition is tuned, so \(\Log \clG\) is locally tabular by Theorem~\ref{thm:LFviaTuned}. By hypothesis, \(\Log \quotfr{\clF}\) is also locally tabular. Then \(\Log\clF\) is locally tabular by Theorem~\ref{thm:sum}.
\end{proof}

\begin{proof}[Direct proof]
  The `if' direction is trivial, since $\Log\clF\subseteq \Log\quotfr{\clF}$ by Lemma \ref{lem:p-morph-sim}.

  `Only if'. 
   Assume that all frames in $\quotfr{\clF}$ are $f$-tunable.  
  Let \(\clV\) be a finite partition of $F$, $n=|\clV|$. 
Let \(\sim\) be defined as in Definition~\ref{def:tilde}. 
For $\sim$-classes $[x],[y]$, put $[x]\approx [y]$ iff 
$$
\AA V\in \clV \, ([x]\cap V\neq \emp   \tiff [y]\cap V\neq \emp )
$$
Readily, $\approx$ has at most $ 2^n$ equivalence classes. Let $\clU$ be the corresponding partition of $\quotfr{\clF}$. 
There is a tuned refinement $\clW$ of $\clU$ of size at most $f(2^n)$. 
Now put $x\equiv y$ iff ($x\sim_\clV y$ and $[x]\sim_\clW[y]$). 

It is straightforward that the quotient of $F$ by this equivalence 
is tuned, and its size does not exceed  $nf(2^n)$. 
\IS{Check the notation; check the details.} 
\end{proof}
}%hide

\begin{lemma}\label{lem:s5-square-quotient}
  Let $F=(X,R_1,R_2)$ be a point-generated $\LS{5}^2$-frame,
  $|\fact{X}{R_1}|=\kappa$, $|\fact{X}{R_2}|=\mu$. 
  Then $\quotfr{F}$ is isomorphic to $\rect{\kappa}{\mu}$.
\end{lemma}
\ISLater{An important variation: 
Can we get prpp-criterion for the clusters in $\LS{4}\times \LS{4}$-frames? 

\hide{
Let $C$ be a cluster in an $\LS{4}\times \LS{4}$-frame, (AND SOMETHING ELSE)
$|\fact{C}{R_1}|=\kappa$, $|\fact{C}{R_2}|=\mu$. 
Then $\quotfr{C}$ is isomorphic to $\rect{\kappa}{\mu}$.

(We need to carefully check if there is an analog in earlier works.)
}
}

\begin{proof}
  Let \(\{H_\alpha\}_{\alpha < \kappa}\) and \(\{V_\beta\}_{\beta < \mu }\) be the elements of \(\fact{X}{R_1}\) and \(\fact{X}{R_2},\) respectively. We claim that \(H_\alpha\cap V_\beta\ne \varnothing\) for any \(\alpha< \kappa,\,\beta < \mu.\) Let \(a \in H_\alpha\) and \(b\in V_\beta.\) Since \(F\) is point-generated, \(b\in R_2[R_1(a)],\) so there exists \(c\in X\) such that \(a R_1 c\) and \(c R_2 b.\) Then \(c \in R_1(a) = H_\alpha\) and \(c\in R_2(b) = V_\beta,\) so \(c\in H_\alpha\cap V_\beta.\) 

  Define \(f:\:\kappa\times \mu \to \fact{X}{\sim}\) by \(f(\alpha,\,\beta) = H_\alpha \cap V_\beta.\) 
  We claim that \(f\) is an isomorphism between \(\rect{\kappa}{\mu}\) and \(\quotfr{F}.\) 
  For any \([a]\in \fact{X}{\sim}\) there exist unique \(H_\alpha,\,V_\beta\) such that \(a\in H_\alpha \cap V_\beta,\) so \([a] = H_\alpha\cap V_\beta,\) thus \(f\) is a bijection. 
  Let $R_H$ and $R_V$ denote the horizontal and vertical relations in \(\rect{\kappa}{\mu}\). If \((\alpha,\beta) R_H (\alpha',\beta')\) then \(\alpha = \alpha'.\) Consider \(a\in H_\alpha\cap V_\beta\) and \(b\in H_\alpha\cap V_\beta'.\) Observe that \(H_\alpha = R_1(a),\) so \(a R_1 b,\) thus \(f(\alpha,\beta) = [a] \quotfr{R_1}[b] = f(\alpha',\beta').\) Conversely, if \(f(\alpha,\beta) \quotfr{R_1} f(\alpha',\beta'),\) there exist \(a\in H_\alpha\cap V_\beta\) and \(b\in H_\alpha'\cap V_\beta'\) such that \(a R_1 b.\) Then \(H_\alpha = R_1(a) = R_1(b) = H_\alpha',\) so \(\alpha = \alpha'\) and \((\alpha,\beta) R_H (\alpha,\beta') = (\alpha',\beta').\) An analogous argument applies for \(R_V.\) %\ISH{$R_H$ and $R_V$ seem to be undefined}
\end{proof}

\extended{
\begin{proposition}\label{prop:rpp-factor}
Assume that in every frame in a class  $\clF$  of $\Al$-frames all relations are transitive. 
Then 
$\clF$ has rpp iff 
$\fact{\clF}{\sim}$ has the rpp. 
\end{proposition}
\begin{proof}
    \ISH{Seems to be trivial. DC, write down}
    \ISH{Seems to be obsolete; move to Store?}
\end{proof}
}

\smallskip 

By Theorem \ref{thm:S5Nick} \cite{NickS5}, all proper extensions of $\LS{5}^2$ are locally tabular. 
Assuming  Kripke completeness of a proper extension $L$ of $\LS{5}^2$, local tabularity is a simple corollary of  
Theorem \ref{thm:criterion-frames-general} and our previous observations.
Indeed, let $\clF$ 
be the class of point-generated frames of $L$.
By Lemma~\ref{lem:p-morph-sim}, $L \subseteq \Log{\quotfr{\clF}}$.  
So $\Log{\quotfr{\clF}}$ is a proper extension of $\LS{5}^2$ as well, and hence \(n = \sup\{m  \mid \exists F\in\clF\  \,(\quotfr{F}\toto \rect{m}{m} )\}\) is finite. 
For \(i = 1,\,2\), let $\clG_i$  
be the class of frames of form $\rect{X_1}{X_2}$ with $|X_i|\leq n$. 
By Lemma~\ref{lem:s5-square-quotient}, 
each frame in  $\quotfr{\clF}$ is isomorphic to a frame in $\clG_1 \cup \clG_2$. 
   Both classes $\clG_i$ are uniformly tunable by Theorem~\ref{thm:criterion-frames-general},
   and so is the class $\quotfr{\clF}$.\footnote{This is a particular case of the following general fact. 
   If the varieties generated by classes $\clC_1$ and $\clC_2$ are locally finite, then the 
   variety generated by  the class  $\clC_1 \cup \clC_2$ is locally finite as well:  $\clC_1 \cup \clC_2$ is uniformly locally finite by Malcev criterion 
   given in Theorem \ref{Malcev73}.}   
   %\cite[Section 14, Theorem 3]{Malcev73}.} 
   Then $L$ is locally tabular by Lemma~\ref{lem:quotient-uniformly-tunable}.

This does not give a complete proof of Theorem \ref{thm:S5Nick}: we need to exclude the case of incomplete logics. 
For a logic $L$, let $F_L$ be its canonical frame \cite[Section 4.2]{BDV}. 
The following reasoning is a variant of the proof given in 
\cite[Claim 4.7]{NickS5};\footnote{We are grateful for an anonymous reviewer for providing the reference to this claim.} we provide it for the self-containment of the text. 
Together with the previous reasoning, it completes the proof of Theorem \ref{thm:S5Nick}.  
\begin{lemma}\cite{NickS5}\label{lem:ext-s5sq-canon}
If  $L$ is a proper extension of $\LS{5}^2$, then so is $\Log{F_L}$.
\end{lemma}
\begin{proof}%[Proof 1]
For the sake of contradiction, assume that \(\Log{F_L} = \LS{5}^2\). We show that in this case $L=\LS{5}^2$.

Let $m<\omega$. By Proposition \ref{prop:Jankov-Fine}, %\ISH{Jankov-Fine for pretransitive frames}
$f:F\toto \rect{m}{m}$ for a point-generated subframe $F=(X,R_1,R_2)$ of $F_L$ and a p-morphism $f$.

The logic $\LS{5}^2$ is canonical \cite{Segerberg2DimML}, \cite{Sheht-TwoDim78}.
It follows that $R_1$ and $R_2$ are equivalences, and for all $a,b\in X$, we have 
\begin{equation}\label{eq:cone}
a (R_1\circ R_2) b \text{ and } a (R_2\circ R_1) b.  
\end{equation}

For $i<m$, choose $d_i\in f^{-1}(i,i)$. 
For distinct $i,j<m$, we have
%\begin{equation}
  $(d_j,d_i)\notin R_1\cup R_2;$
%\end{equation}
hence, there is a formula $\delta_{ij}\in d_i$ such that 
$\Di_1\delta_{ij}\vee \Di_2\delta_{ij}\notin  d_j$. 
For $i>0$, put $\delta_i=\bigwedge_{i\neq  j<m} \delta_{ij}$; 
put $\delta_0=\bigwedge_{0<k<m} \neg (\Di_1\delta_{k}\vee \Di_2\delta_{k})$. 
We claim that
\begin{equation}\label{eq:dij}
\delta_i\in d_i.  
\end{equation}
For $i>0$, this is trivial. To show that $\delta_0\in d_0$, assume that 
$\delta_{k}\in a$ for some $a\in X$ and $k>0$.  Then 
$\delta_{k0}\in a$ and $\neg (\Di_1\delta_{k0}\vee \Di_2\delta_{k0})\in d_0$. 
Hence, $(d_0,a)\notin R_1\cup R_2$. Hence, $\neg (\Di_1\delta_{k}\vee \Di_2\delta_{k})\in d_0$.

Define {\em horizontal} and {\em vertical}  sets: $H_i=\{a\in X\mid \Di_1 \delta_i\in a\}$,
$V_i=\{a\in X\mid \Di_2 \delta_i\in a\}$. 
We have
\begin{equation}\label{eq:lift}
R_1(d_i) \subseteq H_i \text{ and }R_2(d_i) \subseteq V_i. 
\end{equation}
Since $R_1$ and $R_2$ are equivalence relations, we also have 
%\VS{Looks like we need symmetry for \(R_1(H_i) \subseteq H_i\). Are we implicitly using the \(\LS{5}^2\)-consistency?}
\begin{equation}\label{eq:HVmon}
\text{$R_1[H_i]=H_i$ and $R_2[V_i]=V_i$.}
\end{equation}
All sets $H_i$ and $V_i$ are non-empty, since $d_i\in H_i\cap V_i$. 

Let us check that 
%\begin{equation}\label{eq:cover}
$\bigcup_{i<m} H_i=\bigcup_{i<m} V_i=X$.
%\end{equation}
For this, assume that $a\in X$, and $a\notin \bigcup_{0<i<m} V_i$. 
By \eqref{eq:cone}, $d_0R_1bR_2a$ for some $b$. Let $0<k<m$. 
Hence: $\neg \Di_1\delta_{k}\in b$, since $d_0 R_1b$; $\neg \Di_2\delta_{k}\in b$, 
since $bR_2a$; hence, $\delta_0\in b$, and so $a\in H_0$. Likewise, $\bigcup_{i<m} H_i=X$.

Let us check that  
%\begin{equation}\label{eq:disj1}
$H_i \cap H_j =\emp =V_i \cap V_j$ for distinct $i,j$. 
%$\end{equation} 
 Indeed, if $a\in H_i \cap H_j$, then $\Di_1 \delta_i\con \Di_1 \delta_j\in a$, and so $\delta_i\con \Di_1 \delta_j\in b$ for some $b\in R_1(a)$, which implies $i=j$.  Similarly for vertical sets.
It follows that $\{H_i\}_{i<m}$ and $\{V_i\}_{i<m}$ are $m$-element partitions of $X$.

Put $C_{ij}=V_i \cap H_j$.  
%These sets are pairwise disjoint by \eqref{eq:HiVi}. 
By \eqref{eq:cone}, we have $d_i R_2 a R_1 d_j$ for some $a$, and so each $C_{ij}$ is non-empty
according to \eqref{eq:lift}. It follows that $\{C_{ij}\}_{i,j<m}$ is an $m^2$-element partition of $X$.
For $a\in C_{ij}$, put $g(a)=(i,j)$. 
Hence, $g$ maps $X$ onto $m\times m$. 

We claim that $g$ is a p-morphism. 
By \eqref{eq:HVmon}, $g$ is a homomorphism. 
It remains to check the back property. 
Let $a\in C_{ij}$, and assume that in $\rect{m}{m}$, $\tau$ is related to $(i,j)$
by the horizontal  relation. Then $\tau=(k,j)$ for some
$k<m$.
%\VS{Something wrong, is \(k\) just any?}
We have $a R_1 b R_2 d_{k}$ for some $b$ due to \eqref{eq:lift}. Hence, $b\in C_{kj}$ and so $g(b)=(k,j)$.  
The proof for $R_2$ is completely  analogous. 

Assume that a formula $\vf$ is satisfiable in a model $(\rect{m}{m},\theta)$. 
Define the following valuation $\eta$ on $F$: put
$\eta(p)=\bigcup\{C_{ij}\mid (i,j)\in \theta(p)\}$. 
Then $g$ is a p-morphism of $(F,\eta)$ onto $(\rect{m}{m},\theta)$. 
So $\vf$ is satisfiable in $(F,\eta)$. 

We have $(F,\theta_L)\mo L$, where $\theta_L$ is the restriction of the canonical valuation on $F$.  
Let $\psi_p$ be the formula $\bigvee \{\Di_1 \delta_i\con \Di_2 \delta_j \mid (i,j)\in \theta(p) \}$. 
For each $a\in X$, we have $a\in \eta(p) \text{ iff } \psi_p \in a,$
so $(F,\eta)\mo L$. It follows that $L=\LS{5}^2$. 
\end{proof}

\hide{
\begin{proof}
  By contraposition, assume that \(\Log{F_L} = \LS{5}^2.\)

  % Since \(\LS{5}^2 = \Log\{\rect{m}{m}\mid m < \omega\},\) there exist a number \(m < \omega,\) a valuation \(V\) of \(\rect{m}{m}\) and a point \(y\in m\times m\) such that \(\rect{m}{m},\,V,\,y\models \varphi.\)

  % Since \(\Log{F_L} = \LS{5}^2,\,F_L\not\models BCl_m,\) so there exists a point generated subframe \(C=(X,\,R_1,\,R_2)\) of \(F_L\) such that \(\left|\faktor{X}{R_1}\right|>m\) and \(\left|\faktor{X}{R_2}\right|>m.\) Then \(\rect{m}{m}\) is a p-morphic image of \(C.\)

  Fix \(m > 1.\) Let \(\chi_m\) be the Jankov-Fine formula of \(\rect{m}{m}, \) which is the conjunction of the following formulas:
  \begin{enumerate}
     \item \( p_{0,0};\)
     \item \(\bigwedge_{i,\,j<m} \Di_1 \Di_2 p_{i,j};\)
     \item \(\Box_1 \Box_2 \bigvee_{i,\,j < m}p_{i,j};\)
     \item \(\Box_1 \Box_2 \bigwedge_{i,j,j' < m}\left(p_{i,j} \to \Di_1 p_{i,j'}\right);\)
     \item \(\Box_1 \Box_2 \bigwedge_{i,j,j' < m}\bigwedge_{i'\ne i} \left(p_{i,j} \to \lnot\Di_1 p_{i',j'}\right);\)
     \item \(\Box_1 \Box_2 \bigwedge_{i,i',j < m}\left(p_{i,j} \to \Di_2 p_{i',j}\right);\)
     \item \(\Box_1 \Box_2 \bigwedge_{i,i',j < m}\bigwedge_{j'\ne j} \left(p_{i,j} \to \lnot \Di_2 p_{i',j'}\right).\)
   \end{enumerate}
   Since \(\Log{F_L} = \LS{5}^2,\) \(\chi_m\) is satisfiable in \(F_L:\) there exist a valuation \(V_m\) of \(C\) and a point \(c\in \dom F_L\) such that \(F_L,\,V_m,\,c\models V_m.\) Let \(C = F_L \uparrow x\) be the subframe generated by \(c.\)

   Pick representatives \(x_{i,j}\in V(p_{i,j})\) for \(i,\,j< m.\) Then the restriction of \(R_1,\,R_2\) on \(\{x_{i,j}\}_{i,\,j<m}\) forms a frame isomorphic to \(\rect{m}{m}:\)
   \begin{itemize}
     \item Since \(c\models \Box_1 \Box_2 \bigwedge_{i,j,j' < m}\left(p_{i,j} \to \Di_1 p_{i,j'}\right),\,\) and \(x_{i,j}\models p_{i,j},\) it is true \ISH{I do not see why} that \(x_{i,j} R_1 x_{i,j'}\) for any \(i,\,j,\,j'.\) Therefore \(R(x_{i,j}) \supseteq \{x_{i,j'}\}_{j'<m}.\)
     \item Since \(c\models\Box_1 \Box_2 \bigwedge_{i,j,j' < m}\bigwedge_{i'\ne i} \left(p_{i,j} \to \lnot\Di_1 p_{i',j'}\right),\) the negation of \(\Di_1 p_{i',j'}\) is true in \(x_{i,j}\) for any \(i'\ne i\) and any \(j'.\) Then \(R_1(x_{i,j}) \subseteq \{x_{i,j'}\}_{j'<m}.\) Therefore \(R_1(x_{i,j}) = \{x_{i,j'}\}_{j'<m}.\)
     \item Analogously, \(R_2(x_{i,j}) = \{x_{i',j}\}_{i' < m}\) for any \(i,\,j<m.\)
   \end{itemize}

   Find formulas \(\alpha_{i,j}\) that distinguish \(x_{i,\,j}\) as maximal \(L\)-consistent sets:
   \[
     V_L\uparrow C(\alpha_{i',j'})\cap \{x_{i,j}\}_{i,\,j<m} = \{x_{i',j'}\},\quad i',\,j'< m.
   \]
   Note that \(V_L\uparrow C\) is the restriction of the canonical valuation to \(C\), which is different from the valuation \(V_m\uparrow C\) which satisfies \(\chi_m.\)

   Since \(x_{i',j'}\not\in  R_1 (x_{i,j})\) for \(i'\ne i,\) there exists a formula \(\gamma_{i,j,i',j'}\in x_{i',j'}\) such that \(\Di_1 \gamma_{i,j,i',j'}\not \in x_{ij}.\) Analogously, for any \(i,\,j,\,i',\,j'\) with \(j\ne j',\) there is a formula \(\delta_{i,j,i',j'}\in x_{i',j'}\) such that \(\Di_2 \delta_{i,j,i',j'}\not \in x_{i,j}.\)

   For any \(i,\,j < m\) let
   \[
     \Phi_{i,j} = \alpha_{i,j} \land \bigwedge_{(i',j')\ne(i,j)} (\gamma_{i,j,i',j'} \land \delta_{i,j,i',j'}).
   \]
   Then \(V_L(\Phi_{i,j})\cap \{x_{i,j}\}_{i,j < m} = \{x_{i,j}\}.\)

   For any \(i<m-1\)  define
   \begin{gather*}
     \Psi_{i,j} = \Phi_{i,j}\land \bigwedge_{i'\ne i}\bigwedge_{j' < m} \lnot \Di_1 \Phi_{i',j'} \land \bigwedge_{i'<m} \bigwedge_{j'\ne j} \lnot \Di_2 \Phi_{i',j'} ,\quad j < m-1;\\
     \Psi_{i,m-1}= \Di_1 \Psi_{i,0}\land \bigwedge_{j<m-1} \lnot \Psi_{i,j}.
   \end{gather*}

   For \(i = m-1\) the formulas \(\Psi_{i,j}\) are defined differently:
   \begin{gather*}
    \Psi_{m-1,j} = \Di_2 \Psi_{0,j}\land \bigwedge_{i<m-1}\lnot \Psi_{i,j},\quad j < m-1;\\
    \Psi_{m-1,m-1} = \bigwedge_{i,j < m-1} \lnot \Psi_{i,j}\land \bigwedge_{i < m-1}\lnot \Di_1 \Psi_{i,0}\land \bigwedge_{j < m-1} \lnot \Di_2\Psi_{0,j}.
   \end{gather*}

   Let \(f:\:C \to m\times m\) be the map that sends every element of \(V_L\uparrow C(\Psi_{i,j})\) to \((i,j).\) To show that \(f\) is well-defined, we check that every \(x\in \dom C\) belongs to exactly one of \(V_L(\Psi_{i,j})\).

   Let \(x\in \dom C\) and assume \(x\not\in V_L (\Psi_{i,j})\) for any \(i,\,j<m-1.\) If \(\Psi_{i,0}\) is true at \(x\) for some \(i < m-1\) then \(x\in V_L(\Psi_{i,m-1});\) if \(\Psi_{0,j}\) is true at \(x\) for some \(j < m-1\) then \(x\in V_L(\Psi_{m-1,j});\) otherwise, \(x\in V_L(\Psi_{m-1,m-1}).\)

   Assume that \(C,\,V_L,\,x\models \Psi_{i,j}\land \Psi_{i',j'}\) for some \((i,\,j)\ne(i',j').\)
   \begin{itemize}
     \item If \(i = j = m-1\) then \(\Psi_{i',j'}\) is false in \(x\)
     \item If \(j = j' = m-1\) and  \(i,i' < m-1\) then
     \[
       C,\,V_L,\,x\models \Di_1 \Psi_{i,0} \land \Di_1 \Psi_{i',0};
     \]
     note that \(\Psi_{i,0}\) implies \(\lnot \Di_1 \Phi_{i',0}\) and \(\Psi_{i',0}\) implies \(\Di_1 \Phi_{i',0}.\) Then
     \begin{gather*}
        C,\,V_L,\,x\models \Di_1 \lnot \Di_1 \Phi_{i',0} \land \Di_1 \Di_1 \Psi_{i',0}\\
        C,\,V_L,\,x\models \Di_1 \Box_1 \lnot \Phi_{i',0} \land \Di_1 \Psi_{i',0},\\
      \end{gather*}
      so \(R_1(x)\) contains a point \(y\in V_L(\Psi_{i',0})\) and a point \(z\) such that \(R_1(z) \cap V_L(\Psi_{i',0}) = \varnothing.\) But \(R_1\) is an equivalence relation, so \(R_1(x) = R_1(z),\) a contradiction.
      \item Analogously, the case \(i = i' = m-1\) and \(j,\,j' < m-1\) leads to a contradiction since \(R_2\) is an equivalence relation.
     \item Finally, if \(i,\,j,\,i',\,j' < m-1\) then \(\Psi_{i,j}\) implies \(\Phi_{i,j}\) and hence \(\Di_1 \Phi_{i,j}\) by the reflexivity of \(R_1;\) \(\Psi_{i',j'}\) implies \(\lnot \Di_1 \Phi_{i,j},\) a contradiction.
   \end{itemize}

   Then \(f\) is well-defined.

   Let \(V\) be any valuation on \(\rect{m}{m}.\) We construct a new valuation \(V^*\) on \(C\) by setting
   \[
     V^*(p) = \bigcup \{V_L\uparrow C(\Psi_{i,j}) \mid (i,\,j)\in V(p)\}.
   \]
   We claim that \(f\) is a p-morphism of models from \((C,\,V^*)\) to \((\rect{m}{m},\,V).\)

   By the construction,
   \begin{align*}
     f(x)\in V(p) &\iff \exists (i,j):\:f(x) = (i,j)\text{ and }(i,j)\in V(p)\\
     &\iff \exists (i,j):\:x\in V_L\uparrow C(\Psi_{i,j}) \text{ and }(i,j)\in V(p)\\
     &\iff x\in \bigcup \{V_L\uparrow C(\Psi_{i,j}) \mid (i,j)\in V(p)\}\iff x\in V^* (p).
   \end{align*}

   Let \(x,\,y\in \dom C\) and \(x R_1 y.\) Find \(i,\,j\) such that \(x\in V_L(\Psi_{i,j})\) and \(i',\,j'\) such that \(y\in V_L(\Psi_{i',j'}).\) If \((i,\,j)=(i',\,j')\) then \(f(x) = f(y),\) hence \(f(x) R_1 f(y)\)   by reflexivity of \(\rect{m}{m}.\) Then assume \(j\ne j'\) and show that \(i = i'.\) Since \(xR_1 y,\,\) \(x\in V_L( \Psi_{i,j}\land \Di_1 \Psi_{i',j'}).\) If \(i,\,j,\,i',\,j' < m-1\) then \(x\models \lnot \Di_1 \Phi_{i'',j'}\)  for any \(i''\ne i.\) But \(\Di_1 \Psi_{i',j'}\) entails \(\Di_1 \Phi_{i,j},\) so it's only possible that \(i = i'.\) If \(i = m-1\ldots?\)

   % Show that \(R_1\cap R_2(x_{i,j}) = V_L(\Psi_{i,j})\) for \(i,\,j < m-1.\)
   % \begin{description}
   %   \item[\((\subseteq)\)] If \(i'\ne i,\,i'<m-1\) then \(\Psi_{i',j}\) implies \(\lnot \Di_1 \Phi_{i,j}.\) Since \(x_{i,j}\in V_L(\Phi_{i,h}),\) it follows that \(R_1(x_{i,j})\) does not contain points from \(V_L(\Psi_{i',j})\) for \(i'<m-1. \) Analogously, \(R_2(x_{i,j})\) does not contain points from \(V_L(\Psi_{i,j'})\) for \(j' \ne j,\,j' < m-1.\) If
   %   \item[\((\supseteq)\)]
   % \end{description}

   Let \(f(x)=(i,j) R_1 (i',j').\) Then \(x\in V_L(\Psi_{i,j}).\)
\end{proof}
}

Now Theorem \ref{thm:S5Nick} follows: 
if $L$ is a proper extension of $\LS{5}^2$, then 
so is the logic 
$\Log F_{L}$ of the canonical frame of $L$; the latter logic is Kripke complete and so is locally tabular;
since $\Log F_{L}\subseteq L$, $L$ is locally tabular.   

\end{document}